\documentclass[11pt,english]{amsart}
\usepackage[]{fontenc}
\usepackage[latin9]{inputenc}
\usepackage{geometry}
\usepackage{fancyhdr}
\usepackage{verbatim}
\usepackage{amssymb}

\makeatletter

\providecommand{\LyX}{L\kern-.1667em\lower.25em\hbox{Y}\kern-.125emX\@}
\providecommand{\tabularnewline}{\\}

\numberwithin{equation}{section} 
\numberwithin{figure}{section} 
  \@ifundefined{theoremstyle}{\usepackage{amsthm}}{}
  \theoremstyle{plain}
  \newtheorem{thm}{Theorem}[section]
  \theoremstyle{definition}
  \newtheorem{defn}[thm]{Definition}
  \theoremstyle{plain}
  \newtheorem{cor}[thm]{Corollary}
  \theoremstyle{plain}
  \newtheorem{fact}[thm]{Fact}
  \theoremstyle{plain}
  \newtheorem{prop}[thm]{Proposition}
  \theoremstyle{remark}
  \newtheorem*{rem*}{Remark}
  \theoremstyle{remark}
  \newtheorem*{acknowledgement*}{Acknowledgement}
  \theoremstyle{plain}
  \newtheorem{lem}[thm]{Lemma}
  \theoremstyle{remark}
  \newtheorem{rem}[thm]{Remark}
 \theoremstyle{definition}
  \newtheorem{example}[thm]{Example}
  \theoremstyle{plain}
  \newtheorem*{conjecture*}{Conjecture}

\usepackage{fancyhdr}
\usepackage{multirow}
\def\quot{/\!\!/}
\def\tr{\mathsf{tr}}
\def\de{\mathsf{det}}

\def\Spm{\mathrm{Spm}\!}

\makeatother

\makeatother

\usepackage{babel}

\begin{document}

\title{Simultaneous Similarity and Triangularization\\
of Sets of 2 by 2 Matrices}

\author{Carlos A. A. Florentino}

\begin{abstract}
Let $\mathcal{A}=(A_{1},...,A_{n},...)$ be a finite or infinite sequence
of $2\times2$ matrices with entries in an integral domain. We show
that, except for a very special case, $\mathcal{A}$ is (simultaneously)
triangularizable if and only if all pairs $(A_{j},A_{k})$ are triangularizable,
for $1\leq j,k\leq\infty$. We also provide a simple numerical criterion
for triangularization.%
{} 

Using constructive methods in invariant theory, we define a map (with
the minimal number of invariants) that distinguishes simultaneous
similarity classes for non-commutative sequences over a field of characteristic
$\neq2$. We also describe canonical forms for sequences of $2\times2$
matrices over algebraically closed fields, and give a method for finding
sequences with a given set of invariants.

%
{}\vspace{-6mm}

\end{abstract}

\keywords{Simultaneous similarity, Simultaneous triangularization, Invariants
of 2 by 2 matrices. }

\maketitle

\section{Introduction}

\subsection{Motivation}

The properties of a set of square matrices which are invariant under
simultaneous conjugation have been the subject of many investigations.
In the case of a pair of matrices many problems have been solved including
finding criteria for simultaneous similarity, simultaneous triangularization,
existence of common eigenvectors, etc. Analogous problems have been
solved for subalgebras, subgroups or sub-semigroups of square matrices
(see, for example, \cite{Fr,L,RR} and references therein).

Here, we mainly concentrate on the problems of simultaneous similarity,
simultaneous triangularization and canonical forms of countable sets
of $2\times2$ matrices with entries in an arbitrary field, with the
focus on effectively computable solutions (the triangularization results
will hold over integral domains).

Even though we are concerned with $2\times2$ matrices, it is convenient
to describe the setup more generally as follows. Fix integers $m\geq1$,
$n\geq1$ and let $V_{m,n}$ be the vector space of $n$-tuples of
$m\times m$ matrices with entries in a field $\mathbb{F}$. The group
of invertible matrices $G:=GL_{m}(\mathbb{F})$ acts on $V_{m,n}$
by simultaneous conjugation on every component: \begin{equation}
g\cdot\mathcal{A}:=(gA_{1}g^{-1},...,gA_{n}g^{-1}),\label{eq:accao}\end{equation}
where $g\in G$ and $\mathcal{A}=(A_{1},...,A_{n})\in V_{m,n}$. By
analogy with the case $n=1$, the orbit $G\cdot\mathcal{A}:=\{g\cdot\mathcal{A}:g\in G\}\subset V_{m,n}$
will be called the \emph{conjugacy class} of $\mathcal{A}$, and can
be viewed as an element $[\mathcal{A}]\in V_{m,n}/G$ of the quotient
space. One can consider the following problems.

\newcounter{art}\renewcommand{\labelenumi}{(\roman{art})}\addtocounter{art}{1}

\begin{enumerate}
\item Identify one element for each conjugacy class in a natural way, i.e,
list all possible `canonical forms';\addtocounter{art}{1} 
\item Construct invariants that distinguish all conjugacy classes.\addtocounter{art}{1}
%
{}
\end{enumerate}
Naturally, these problems are related. In general, a solution to (i)
will lead to a solution of (ii); however, the answer obtained in this
way might be unnatural, and the description in terms of invariants
is often useful. %
{}

These are difficult questions in general. Here, we will be concerned
with the description and classification of conjugacy classes in the
case $m=2$. As we will see, in this case the questions above admit
simple and complete answers, given by explicitly computable numerical
criteria. In order to state our results in concise terms (see Definition
\ref{def:subsequence}) we will adopt the following terminology for
(ordered) sets of matrices.

\begin{defn}
\label{def:matrix-sequence}An element $\mathcal{A}=(A_{1},...,A_{n})\in V_{m,n}$
will be called a \emph{matrix sequence of size} $m\times m$ \emph{and
length $n$}, or just a \emph{matrix sequence} when the integers $m$
and $n$ are assumed, and the matrices $A_{1},...,A_{n}$ will be
called the \emph{components} or \emph{terms} of $\mathcal{A}$. We
say that two matrix sequences $\mathcal{A}=(A_{1},...,A_{n})$ and
$\mathcal{B}=(B_{1},...,B_{n})$ in $V_{m,n}$ are \emph{similar}
and write $\mathcal{A}\sim\mathcal{B}$ if they lie in the same conjugacy
class (or $G$ orbit) i.e, if there is an element $g\in G$ such that
$\mathcal{B}=g\cdot\mathcal{A}$ (equivalently, $B_{j}=gA_{j}g^{-1}$
for all $j=1,...,n$). 
\end{defn}
We will also allow sequences $\mathcal{A}=(A_{1},...,A_{n},...)\in V_{m,\infty}$
with a countable infinite number of terms.

It turns out that a concrete answer to the above problems requires
the separation of all matrix sequences (of fixed size and length)
into $G$-invariant subclasses with distinct algebraic and/or geometric
properties. For instance, for the class of matrices that pairwise
commute some of these questions are easier. This case can be reduced
to the case of a single matrix ($n=1$) which is classically solved
with the Jordan decomposition theorem (at least when $\mathbb{F}$
is algebraically closed and some term is nonderogatory).

One property that is often relevant is that of simultaneous triangularization
(or block triangularization); it generalizes commutativity and has
a natural geometric interpretation. To be precise, let us consider
the following standard notions. Note that a sequence $\mathcal{A}\in V_{m,n}$
can be viewed either as an ordered set of $m\times m$ matrices, as
in the definition above, or alternatively, as a \emph{single} $m\times m$
matrix whose entries are $n$-dimensional vectors. \begin{equation}
\mathcal{A}=\left(\begin{array}{ccc}
a_{11} & \cdots & a_{1m}\\
\vdots & \ddots & \vdots\\
a_{m1} & \cdots & a_{mm}\end{array}\right),\quad a_{jk}\in\mathbb{F}^{n}.\label{eq:nmatrix}\end{equation}

\begin{defn}
Let $\mathcal{A}$ be a matrix sequence. $\mathcal{A}$ will be called
\emph{commutative} if all its terms pairwise commute (and non-commutative
otherwise). We will say that $\mathcal{A}$ is an \emph{upper triangular
matrix sequence} if all the vectors below the main diagonal are zero
($a_{jk}=0$ for all $j>k$ in equation (\ref{eq:nmatrix})), and
that $\mathcal{A}$ is \emph{triangularizable} if it is similar to
an upper triangular matrix sequence. 
\end{defn}
In geometric terms, $\mathcal{A}$ is triangularizable if and only
if there is a full flag of vector subspaces of $\mathbb{F}^{m}$ which
is invariant under every term $A_{j}$ of $\mathcal{A}$ (for the
standard action of $G$ on $\mathbb{F}^{m}$). Observe also that upper
and lower triangularization are equivalent (over any ring). For the
reasons mentioned, we will add the following \emph{triangularization
problem} to the previous list.

\begin{enumerate}
\item Give an effective numerical criterion for a matrix sequence to be
triangularizable.
\end{enumerate}
Problem (iii) was solved effectively by the following refinement of
McCoy's Theorem (\cite{Mc}). Let $[A,B]=AB-BA$ denote the commutator
of two matrices.

\begin{thm}
\label{thm:McCoy}\emph{{[}McCoy, see \cite{Mc},\cite{L}]} Let $\mathbb{F}$
be algebraically closed. A $m\times m$ matrix sequence $\mathcal{A}$
of length $n$ is triangularizable if and only if\[
p(A_{1},...,A_{n})[A,A']\]
is nilpotent for all monomials $p(x)$ (in non-commuting indeterminants
$x_{1},...,x_{n}$) of total degree not greater than $m^{2}$ and
all terms $A,A'$ of $\mathcal{A}$. 
\end{thm}
%
{}Using results of Paz \cite{Paz} and Pappacena \cite{Pap}, the bound
on the degree $d$ of the monomials $p(x)$ in Theorem \ref{thm:McCoy}
can be improved to $d\leq\frac{m^{2}}{3}+1$ and $d\leq m\sqrt{\frac{2m^{2}}{m-1}+\frac{1}{4}}+\frac{m}{2}-1$,
respectively, each one being more efficient for smaller (resp. larger)
values of $m$.

To discuss simultaneous similarity, consider the notion of subsequence.

\begin{defn}
\label{def:subsequence}A \emph{subsequence} of a matrix sequence
$\mathcal{A}=(A_{1},...,A_{n})\in V_{m,n}$ is a matrix sequence of
the form \[
\mathcal{A}_{J}=(A_{j_{1}},...,A_{j_{l}})\in V_{m,l}\]
obtained from $\mathcal{A}$ by deleting some of its terms (here $J=(j_{1},...,j_{l})\in\{1,...,n\}^{l}$
for some natural number $l\leq n$, and the indices are strictly increasing:
$1\leq j_{1}<...<j_{l}\leq n$). 
\end{defn}
%
{}Let $\mathbb{A}=\mathbb{F}[A_{1},...,A_{n}]$ be the algebra generated
over $\mathbb{F}$ by the terms of $\mathcal{A}$ (of dimension $\leq m^{2}$).
Suppose we rearrange the terms of $\mathcal{A}$ such that $\mathbb{A}$
is in fact generated by just the first $k\leq m^{2}$ elements. Then,
for all $j>k$, $A_{j}=p_{j}(A_{1},...,A_{k})$ for some polynomial
$p_{j}$ with coefficients in $\mathbb{F}$. From this one can easily
obtain a similarity test.

\begin{fact}
Let $\mathcal{A},\mathcal{B}$ be two $m\times m$ matrix sequences
of the same finite or infinite length $n$. Then $\mathcal{A}$ and
$\mathcal{B}$ are similar if and only if \emph{all corresponding
subsequences} of length $\leq m^{2}$ are similar (that is, $\mathcal{A}_{J}\sim\mathcal{B}_{J}$
for all $J=(j_{1},...,j_{l})$ with $1\leq j_{1}<...<j_{l}\leq n$
and $l\leq m^{2}$). 
\end{fact}
One can improve this statement if we restrict to a big class of matrix
sequences. The Dubnov-Ivanov-Nagata-Higman \cite{DINH} theorem states
that, given $m\in\mathbb{N}$, there is a natural number $N(m)$ with
the following property: every associative algebra over $\mathbb{F}$
satisfying the identity $x^{m}=0$ is nilpotent of degree $N(m)$
(i.e, any product $x_{1}\cdots x_{k}$ of $k$ elements $x_{1},...,x_{k}$
of the algebra is zero, for $k\geq N(m)$). It is known that $\frac{m(m+1)}{2}\leq N(m)\leq\min\left\{ 2^{m}-1,m^{2}\right\} $
\cite{Ra}, and it was conjectured that $N(m)=\frac{m(m+1)}{2}$.
This was verified to be true for $m=2,3,4$.

The following remarkable result of  Procesi (\cite{Pr}) relates invariants
of matrices and nilpotency degrees of associative nil-algebras. It
provides explicit generators for the algebra of $G$-invariant regular
(i.e, polynomial) functions on $V_{m,n}$. For a multiindex $J=(j_{1},...,j_{k})\in\{1,...,n\}^{k}$
of \emph{length} $|J|=k$, define the $G$-invariant regular function
$t_{J}:V_{n}(\mathbb{F})\to\mathbb{F}$ as the trace of the word in
the terms of $\mathcal{A}\in V_{n}(\mathbb{F})$ dictated by $J$,
that is \begin{equation}
t_{J}(\mathcal{A}):=\tr(A_{j_{1}}...A_{j_{k}}).\label{eq:generators}\end{equation}

\begin{thm}
\label{thm:Procesi}\emph{{[}Procesi, \cite{Pr}]} Let $\mathbb{F}$
have characteristic $0$, and let $f:V_{m,n}\to\mathbb{F}$ be a $G$-invariant
polynomial function. Then $f$ is a polynomial in the set of functions
\[
\{t_{J}:|J|\leq N(m)\}\]
where $|J|$ denotes the length of the multiindex $J$.
\end{thm}
Let us say that a matrix sequence $\mathcal{A}$ is \emph{semisimple}
if its $G$ orbit is Zariski closed in $V_{m,n}$. Semisimple sequences
form a dense subset of $V_{m,n}$. Noting that $N(2)=3$, Procesi's
theorem has the following corollary.

\begin{cor}
\label{cor:Procesi}Let $\mathcal{A},\mathcal{B}$ be two $2\times2$
semisimple matrix sequences of the same finite or infinite length.
Then $\mathcal{A}\sim\mathcal{B}$ if and only if \emph{all corresponding
subsequences} of length $\leq3$ are similar (that is, $\mathcal{A}_{J}\sim\mathcal{B}_{J}$
for all $J$ with length $\leq3$). 
\end{cor}
This corollary gives a numerical criterion for simultaneous similarity
of semisimple sequences. In the case $m=2$, the number of generators
can be further reduced, and moreover, Drensky described all relations
in terms of a generating set (\cite{D}). He states his main theorem
for traceless matrices, but an easy modification yields:

\begin{thm}
\emph{\label{thm:Drensky}{[}Drensky, \cite{D}]} The generators of
the algebra (over $\mathbb{F}$ of characteristic 0) of $G$-invariant
polynomial functions on $V_{2,n}$ are given by:\[
t_{j},\ t_{jj},\ u_{jk},\ s_{jkl}\quad\quad1\leq j<k<l\leq n,\]
where $u_{jk}:=2t_{jk}-t_{j}t_{k}$, $s_{jkl}:=t_{jkl}-t_{lkj}$ and
$t_{j},t_{jk},t_{jkl}$ are as in (\ref{eq:generators}). A full set
of relations is\begin{equation}
s_{abc}s_{def}+\frac{1}{4}\left|\begin{array}{ccc}
u_{ad} & u_{ae} & u_{af}\\
u_{bd} & u_{be} & u_{bf}\\
u_{cd} & u_{ce} & u_{cf}\end{array}\right|=0,\quad\quad\quad u_{ea}s_{bcd}-u_{eb}s_{acd}+u_{ec}s_{abd}-u_{ed}s_{abc}=0,\label{eq:relations}\end{equation}
for all appropriate indices.
\end{thm}
In this article, we show that these same generators can be used to
get even more explicit solutions to problems (i)-(iii) in the case
$m=2$. Our methods are mostly elementary and their generalization
to higher $m$ seems possible, although computationally very involved.%
{}

\subsection{Statement of the main results\label{sub:Statement}}

From now on, we restrict to the space $V_{n}=V_{2,n}$ of sequences
of $2\times2$ matrices of length $n\in\mathbb{N}\cup\left\{ \infty\right\} $,
i.e, the case $m=2$, except where explicitly stated.

The article can be roughly divided into two parts. In the first part,
Section \ref{sec:Triangularization}, we work with the space $V_{n}(R)$
of $2\times2$ matrix sequences with coefficients in an integral domain
$R$, on which the group $G=GL_{2}(R)$ of invertible matrices over
$R$ acts by conjugation. Here, we define and study reduced sequences,
consider the triangularization problem, and prove Theorems \ref{thm:Flo}
to \ref{thm:efic}, stated below. They provide an efficient numerical
criterion for triangularization of matrix sequences in $V_{n}(R)$. 

The case $m=2$ is simple enough that some easy arguments already
improve some of the statements above, even in the more general case
of $2\times2$ matrices over integral domains. 

\begin{prop}
\label{pro:leq3}$\mathcal{A}\in V_{n}(R)$ is triangularizable if
and only if every subsequence of length $\leq3$ is triangularizable.
\end{prop}
For simultaneous similarity, a simple argument generalizes and improves
Corollary \ref{cor:Procesi}, in the case $m=2$, to account for all
matrix sequences (not necessarily semisimple).

\begin{prop}
\label{pro:sim3}Over an integral domain $\mathcal{A}\sim\mathcal{B}$
if and only if $\mathcal{A}_{J}\sim\mathcal{B}_{J}$ for all $J$
of length $\leq3$. Moreover, under the generic condition $[A_{1},A_{2}]\neq0$,
$\mathcal{A}$ is similar to $\mathcal{B}$ if and only if $(A_{1},A_{2},A_{j})\sim(B_{1},B_{2},B_{j})$,
for all $j=3,...,n$.
\end{prop}
In other words, the conjugation action of $GL_{2}(R)$ on $2\times2$
matrix sequences of any length is completely determined by the same
action on triples of $2\times2$ matrices. These two propositions
seem to be standard, and I thank R. Guralnick for furnishing simple
arguments leading to their proofs, included below for convenience
and completeness.

The statement of Proposition \ref{pro:leq3} is the best possible
in this generality, as there are $2\times2$ matrix sequences of length
$3$ (therefore, \emph{a fortiori} for every size $m\times m$, $m\geq2$
and every length $n\geq3$) that are not triangularizable but are
pairwise triangularizable (see Example \ref{exa:example} below).
See however Theorem \ref{thm:leq2} below.

Let us now consider the problem of finding a numerical criterion for
triangularization of general $2\times2$ sequences. By Proposition
\ref{pro:leq3}, we just need to consider \emph{triples} of $2\times2$
matrices. Using only the $G$-invariant functions given by the trace
and the determinant, the following is a simple numerical criterion.

\begin{thm}
\label{thm:Flo}A $2\times2$ matrix sequence $\mathcal{A}$ (over
$R$) of length $n$ is triangularizable if and only if all its terms
are triangularizable and \begin{equation}
\de(AB-BA)=\mathsf{tr}(ABC-CBA)=0\label{eq:Flo}\end{equation}
for all terms $A,B,C$ of $\mathcal{A}$. In particular, a pair $(A,B)\in V_{2}(R)$
is triangularizable if and only if both $A$ and $B$ are triangularizable
and $\de(AB-BA)=0$.
\end{thm}
\begin{rem*}
Over an algebraically closed field $\bar{\mathbb{F}}$, every single
matrix is triangularizable. Suppose $\mathcal{A}$ is a $2\times2$
matrix sequence over $\bar{\mathbb{F}}$ with $[A,B]$ and $C[A,B]$
nilpotent for all terms $A,B,C$ of $\mathcal{A}$. Then, Equation
(\ref{eq:Flo}) holds and Theorem \ref{thm:Flo} shows that all mentioned
bounds in Theorem \ref{thm:McCoy} can be improved for $m=2$, as
the condition that $p$ is a monomial of degree $\leq1$ is already
sufficient (and necessary) for triangularization.
\end{rem*}
Theorem \ref{thm:Flo} generalizes a result proved in \cite{Fl} for
algebraically closed fields. Observe that the case $n=2$ of this
theorem is a direct generalization to integral domains of a well-known
criteria, obtained in \cite{Fr}, for a pair of $2\times2$ matrices
over $\bar{\mathbb{F}}$ to be triangularizable. 

As a consequence of our study of invariants, we can improve Proposition
\ref{pro:leq3} under a simple non-degeneracy condition on $\mathcal{A}$.
Let us define a \emph{reduced} sequence to be one with no commuting
pairs among its terms. 

\begin{thm}
\label{thm:leq2}A $2\times2$ reduced matrix sequence $\mathcal{A}$
of length $\geq4$ is triangularizable (over $R$) if and only if
all subsequences of $\mathcal{A}$ of length $\leq2$ are triangularizable
(over $R$).
\end{thm}
With a little more care we get a test, computationally much more efficient,
whose complexity grows \emph{only linearly} with the number $n$ of
matrices in $\mathcal{A}$. Let us define the reduced length of a
sequence $\mathcal{A}$ to be the biggest length of a reduced subsequence
$\mathcal{B}\subset\mathcal{A}$.

\begin{thm}
\label{thm:efic}Let $\mathcal{A}=(A_{1},...,A_{n})\in V_{n}$ have
reduced length $l\le n$ and rearrange its terms so that $\mathcal{A}'=(A_{1},...,A_{l})$
is reduced. Then

(i) In the case $l\leq3$, $\mathcal{A}$ is triangularizable if and
only if $\mathcal{A}'$ is triangularizable.

(ii) In the case $l\geq4$, $\mathcal{A}$ is triangularizable if
and only if $A_{1},...,A_{l}$ are triangularizable and $\de([A_{j},A_{k}])=0$
for all $j=1,2,3$ and all $j<k\leq l$.
\end{thm}
In the second part of the article, we work mainly over a field $\mathbb{F}$.
Section \ref{sec:SimSim} deals with simultaneous similarity for $2\times2$
matrix sequences. Using standard invariant theory, one sees that the
values of the Drensky generators are enough to distinguish semisimple
conjugacy classes. But these classes should depend on $4n-3$ parameters
only, the {}``dimension'' of the quotient space $V_{n}/G$ ($V_{n}$
has dimension $4n$, and $G$ acts as $SL_{2}(\mathbb{F})$, a three
dimensional group) which is much less than the number, $2n+\binom{n}{2}+\binom{n}{3}$,
of Drensky generators. After describing rational invariants that distinguish
general triangularizable sequences, we obtain a solution, with the
minimal number of invariants, to problem (ii) for non-commutative
sequences as follows.

Let $\mathcal{S}'$ (resp. $\mathcal{U}'$) denote the subsets of
$V_{n}=V_{n}(\mathbb{F})$ of semisimple (resp. triangularizable)
sequences such that $A_{1}$ is diagonalizable and $[A_{1},A_{2}]\neq0$.
Using the $G$-invariant functions $t_{J}:V_{n}\to\mathbb{F}$ in
(\ref{eq:generators}) define the maps $\Phi':\mathcal{S}'/G\to\mathbb{F}^{4n-3}$
and $\Psi':\mathcal{U}'/G\to\mathbb{F}^{2n}\times\mathbb{P}^{n-2}$,
where $\mathbb{P}^{k}$ denotes the projective space over $\mathbb{F}$
of dimension $k$, by the formulae \begin{eqnarray}
\Phi'([\mathcal{A}]) & = & \left(t_{1},t_{11},t_{2},t_{22},t_{12},...,t_{k},t_{1k},t_{2k},s_{12k},...,t_{n},t_{1n},t_{2n},s_{12n}\right),\label{eq:Phi'}\\
\Psi'([\mathcal{A}]) & = & \left(t_{1},t_{11},t_{2},t_{12},...,t_{k},t_{1k},...,t_{n},t_{1n};\ \psi'\right),\label{eq:Psi'}\end{eqnarray}
where $\psi'$ is defined later in Section \ref{sec:CanonicalForms},
and $[\mathcal{A}]$ denotes the conjugacy class of $\mathcal{A}$. 

\begin{thm}
\label{thm:Invariants}Let $\mathbb{F}$ be a field of characteristic
$\neq2$. The map $\Phi':\mathcal{S}'/G\to\mathbb{F}^{4n-3}$ is injective
and the map $\Psi':\mathcal{U}'/G\to\mathbb{F}^{2n}\times\mathbb{P}^{n-2}$
is two-to-one.
\end{thm}
The last section (Section \ref{sec:CanonicalForms}) concerns the
classification of canonical forms for sets of $2\times2$ matrices
over $\bar{\mathbb{F}}$ and proposes a solution for problem (i).
The main result is Theorem \ref{thm:canonical}, where five types
of canonical forms are obtained for sequences with at least one non-commuting
pair. It shows that, for non-commutative sequences, the restriction
to $\mathcal{S}'$ and $\mathcal{U}'$ in Theorem \ref{thm:Invariants}
is only apparent (see also Remark \ref{rem:NoRestriction}). We also
describe a simple method for finding a sequence in canonical form
with a given value of $\Phi'$ or $\Psi'$.

%
{}Appendix A contains results on the triangularization a single $2\times2$
matrix over $R$ which are crucial for Theorem \ref{thm:Flo}, and
Appendix B, for completeness, describes the well-known canonical forms
of commuting matrices over $\bar{\mathbb{F}}$. 

%
{}

\begin{acknowledgement*}
I would like to thank J. Dias da Silva whose interesting questions
were at the genesis of this article, and R. Guralnick for calling
my attention to related work in the literature and for providing simplified
proofs of some statements in a previous version. I thank also my colleagues
J. Mourão, J. P. Nunes and S. Lawton for many interesting and motivating
conversations on this and related topics. Work partially supported
by the CAMGSD, IST, Technical Univ. of Lisbon, and {}``Fundação para
a Ciência e a Tecnologia'' through the programs Praxis XXI, POCI/MAT/58549/2004
and FEDER. Some computations performed using CoCoA; document typeset
using \LyX{}. 
\end{acknowledgement*}
%
{} %
{}

\section{Simultaneous Triangularization\label{sec:Triangularization}}

Throughout the article, $R$ will stand for an integral domain, $\mathbb{F}$
for a field and $\bar{\mathbb{F}}$ for an algebraically closed field.
$V_{n}$ (resp. $G$) will denote the space of matrix sequences of
length $n\in\mathbb{N}\cup\left\{ \infty\right\} $ (resp. the group
of invertible $2\times2$ matrices) over the appropriate ring or field.
When the coefficients need to be explicitly mentioned, we will use
the notations $V_{n}(R),$ $G(R)$, etc. Sequences of length $1,2,3$
and $4$ will be called \emph{singlets}, \emph{pairs}, \emph{triples}
and \emph{quadruples}, respectively.

\subsection{Simultaneous triangularization and subtriples; reduced sequences}

We start by fixing notation and recalling some well known facts about
matrices over $\mathbb{F}$ and $R$. After this, we define \emph{reduced
sequences}, a notion which will be fundamental in the sequel. 

For a given $\mathcal{A}\in V_{n}(R)$, a matrix sequence of length
$n\in\mathbb{N}$, we will use the notation\begin{eqnarray*}
\mathcal{A}=(A_{1},...,A_{n}) & = & \left(\begin{array}{cc}
a & b\\
c & d\end{array}\right),\quad a,b,c,d\in R^{n},\\
A_{j} & = & \left(\begin{array}{cc}
a_{j} & b_{j}\\
c_{j} & d_{j}\end{array}\right),\end{eqnarray*}
and we let $e=(e_{1},...,e_{n})$ denote the $n$-tuple $a-d\in R^{n}$.
To avoid the most trivial case, we consider only matrix sequences
with at least one non-scalar term. 

%
{}The commutator of 2 matrices $A_{1}$ and $A_{2}$ is given by \begin{equation}
[A_{1},A_{2}]=\left(\begin{array}{cc}
b_{1}c_{2}-c_{1}b_{2} & e_{1}b_{2}-b_{1}e_{2}\\
c_{1}e_{2}-e_{1}c_{2} & c_{1}b_{2}-b_{1}c_{2}\end{array}\right).\label{eq:comuta}\end{equation}
For later use, record the following straightforward but useful lemma.

\begin{lem}
\label{lem:comm} Let $\mathcal{A}=(A_{1},A_{2})\in V_{2}(R)$ and
$A_{1}$ be a non-scalar matrix. If $A_{1}$ is upper triangular,
then $[A_{1},A_{2}]=0$ if and only if $A_{2}$ is also upper triangular
and \begin{equation}
b_{1}e_{2}-e_{1}b_{2}=0.\label{eq:comm}\end{equation}
Similarly, let $A_{1}$ be diagonal non-scalar. Then $[A_{1},A_{2}]=0$
if and only if $A_{2}$ is also diagonal. 
\end{lem}
\begin{proof}
Suppose $[A_{1},A_{2}]=0$ with $c_{1}=0$. Then we have $b_{1}e_{2}-e_{1}b_{2}=b_{1}c_{2}=e_{1}c_{2}=0$,
using (\ref{eq:comuta}). Since $A_{1}$ is non-scalar either $e_{1}$
or $b_{1}$ is non-zero. In an integral domain, this implies $b_{1}e_{2}-e_{1}b_{2}=c_{2}=0$.
The other statement is similar. 
\end{proof}
This implies the following well known result. Note that a $2\times2$
matrix is non-scalar if and only if it is nonderogatory.

\begin{prop}
\label{cor:comm}Let $\mathcal{A}$ be a commutative matrix sequence
(i.e, all terms pairwise commute) of finite or infinite length over
an integral domain $R$. Then $\mathcal{A}$ is triangularizable if
and only if one of its non-scalar terms is triangularizable. Similarly,
$\mathcal{A}$ is diagonalizable if and only if one of its non-scalar
terms is diagonalizable.\hfill{}$\square$
\end{prop}
As a consequence, the proof of Proposition \ref{pro:leq3} is complete
after the following. Let us denote by $\mathbb{A}$ the algebra generated
over $\mathbb{F}$, the field of fractions of $R$, by the terms of
$\mathcal{A}$. A well known result is that $\mathcal{A}$ is commutative
if and only if the dimension of $\mathbb{A}$ is $\leq2$.

\begin{prop}
Let $n\geq2$ and $\mathcal{A}\in V_{n}(R)$ be a non-commutative
matrix sequence. Then $\mathcal{A}$ is triangularizable if and only
if all subsequences of $\mathcal{A}$ of length $\leq3$ are triangularizable. 
\end{prop}
\begin{proof}
If $\mathcal{A}$ is triangularizable, it is clear that all subsequences
of $\mathcal{A}$ are also triangularizable. Conversely, since $\mathcal{A}$
is non-commutative, we can assume without loss of generality, that
$[A_{1},A_{2}]\neq0$, and that, after a suitable conjugation, both
$A_{1}$ and $A_{2}$ are upper triangular. By hypothesis, all triples
of the form $(A_{1},A_{2},A_{k})$, $k=3,...,n$ are triangularizable.
Then, the algebra $\mathbb{A}_{k}$ generated by $\mathcal{A}_{k}=(A_{1},A_{2},A_{k})$
equals the one generated by $(A_{1},A_{2})$, since one is a subset
of the other but both are three dimensional over $\mathbb{F}$. Indeed,
if $\mathbb{A}_{k}$ was of dimension $\leq2$, $\mathcal{A}_{k}$
would be commutative, and if $\mathbb{A}_{k}$ was four dimensional,
$\mathcal{A}_{k}$ would not be triangularizable. Therefore, for all
$j=3,...,n$, $A_{j}=p_{j}(A_{1},A_{2})$ for some polynomial $p_{j}$
with coefficients in $\mathbb{F}$ and thus, $A_{j}$ is also upper
triangular, for all $j$. 
\end{proof}
Recall also the following.

\begin{lem}
\label{lem:non-der}Let $\mathcal{A}=(A_{1},A_{2})\in V_{2}(\mathbb{F})$
be commutative and $A_{1}$ be a non-scalar matrix. Then $A_{2}=p(A_{1})$
for some degree 1 polynomial $p(x)\in\mathbb{F}[x]$. 
\end{lem}
\begin{proof}
From Equation (\ref{eq:comuta}), the conditions $[A_{1},A_{2}]=0$
can be written as $Mu=0$, where \begin{equation}
M=\left(\begin{array}{rrr}
0 & c_{1} & -e_{1}\\
-c_{1} & 0 & b_{1}\\
e_{1} & -b_{1} & 0\end{array}\right),\quad\quad u=(b_{2},e_{2},c_{2}).\label{eq:comuta2}\end{equation}
Note that $M$ has rank exactly 2, since $A_{1}$ is non-scalar. So,
the vector $u=(b_{2},e_{2},c_{2})$ is in the nullspace of $M$, which
is generated by $(b_{1},e_{1},c_{1})\neq0$. Thus, there is an $\alpha\in\mathbb{F}$
such that $(b_{2},e_{2},c_{2})=\alpha(b_{1},e_{1},c_{1})$. Then,
$A_{2}=(d_{2}-\alpha d_{1})I+\alpha A_{1}$ is, explicitly, the required
polynomial ($I$ denotes the identity $2\times2$ matrix).
\end{proof}
%
{}It is clear that a single matrix is triangularizable over an algebraically
closed field, but not necessarily so over a general integral domain
or field. In Appendix A we include a short account of the conditions
for triangularization of a single matrix in $V_{1}(R)$. The following
notion will play a central role. If $\mathcal{A}$ is a subsequence
of $\mathcal{B}$, we will write $\mathcal{A}\subseteq\mathcal{B}$%
{}.

\begin{defn}
\label{def:reduced}A matrix sequence with at least one non-scalar
term $\mathcal{A}=(A_{1},...)\in V_{n}(R)$ is called \emph{reduced}
if there are no commuting pairs among its terms, that is, $[A_{j},A_{k}]\neq0$,
for all $1\leq j<k\leq n$. We say that $\mathcal{A}$ is a \emph{reduction}
of $\mathcal{B}$ if $\mathcal{A}$ is reduced and is obtained from
$\mathcal{B}$ by deleting some of its terms. Finally, we say that
$\mathcal{A}$ is a \emph{maximal reduction} of $\mathcal{B}$, and
that $l$ is its \emph{reduced length}, if $\mathcal{A}$ is a reduction
of $\mathcal{B}$ of length $l$, and any subsequence $\mathcal{A}'\subseteq\mathcal{B}$
with length $>l$ is not reduced.
\end{defn}
Over a field, by Lemma \ref{lem:non-der}, a reduced sequence $\mathcal{A}$
is one where no term is a polynomial function of another (so all terms
generate an algebra of dimension $2$, and no two terms generate the
same subalgebra of the full matrix algebra). It is clear that any
two maximal reductions have the same length. Note also that any subsequence
of a reduced sequence is also reduced. The following facts show that
important properties like existence of a triangularization are captured
by any maximal reduction of a matrix sequence.

\begin{prop}
\label{pro:reduce} Let $\mathcal{A}=(A_{1},...,A_{n})\in V_{n}(R)$
and let $A_{n+1}$ commute with at least one non-scalar term of $\mathcal{A}$.
Then $\mathcal{A}$ is triangularizable if and only if $\mathcal{A}':=(A_{1},...,A_{n},A_{n+1})$
is triangularizable. 
\end{prop}
\begin{proof}
Naturally if $\mathcal{A}'$ is triangularizable, $\mathcal{A}$ is
also. For the converse, without loss of generality assume $[A_{1},A_{n+1}]=0$
with $A_{1}$ non-scalar and $\mathcal{A}$ in upper triangular form.
Then, by Lemma \ref{lem:comm}, $A_{n+1}$ is also upper triangular,
so that $\mathcal{A}'$ is triangularizable.
\end{proof}
\begin{cor}
\label{cor:reduce}If $\mathcal{A}$ and $\mathcal{A}'$ are arbitrary
sequences with a common maximal reduction (of length $\geq1$), then
either they are both triangularizable or both not triangularizable. 
\end{cor}
\begin{proof}
Let $\mathcal{B}$ be such a common maximal reduction. Then $\mathcal{A}$
and $\mathcal{A}'$ are obtained from $\mathcal{B}$ by adding scalar
matrices or matrices that commute with some of the non-scalar terms
of $\mathcal{B}$. So, if $\mathcal{B}$ is triangularizable, both
$\mathcal{A}$ and $\mathcal{A}'$ are triangularizable by repeatedly
applying Proposition \ref{pro:reduce}. The case when $\mathcal{B}$
is not triangularizable is analogous. 
\end{proof}

\subsection{Necessary conditions for triangularization via invariants}

%
{}We continue to work over an integral domain $R$. Define the following
important $GL_{2}(R)$-invariant functions. For a matrix $A\in V_{1}$,
let $\delta_{A}$ denote the discriminant of its characteristic polynomial,
that is $\delta_{A}=\tr^{2}A-4\de A$.%
{}

\begin{defn}
\label{def:tau,sigma}Let $\tau,\sigma:V_{2}(R)\to R$ and $\Delta:V_{3}(R)\to R$
be defined by \begin{eqnarray*}
\tau(A,B) & := & 2\tr(AB)-\tr(A)\tr(B),\\
\sigma(A,B) & := & \de(AB-BA)\\
\Delta(A,B,C) & := & \left(\tr(ABC-CBA)\right)^{2}.\end{eqnarray*}

\end{defn}
When a matrix sequence is written as $\mathcal{A}=(A_{1},...,A_{n})\in V_{n}(R)$
we will also use \begin{eqnarray*}
\tau_{jk} & = & \tau_{jk}(\mathcal{A})=\tau(A_{j},A_{k})\\
\sigma_{jk} & = & \sigma_{jk}(\mathcal{A})=\sigma(A_{j},A_{k})\\
\Delta_{jkl} & = & \Delta_{jkl}(\mathcal{A})=\Delta(A_{j},A_{k},A_{l}),\end{eqnarray*}
for any indices $j,k,l\in\{1,...,n\}$. By simple computations , we
can express these functions in terms of $b_{j},c_{j},e_{j}$ as follows.\begin{eqnarray}
\tau_{jk} & = & e_{j}e_{k}+2b_{j}c_{k}+2c_{j}b_{k}\nonumber \\
\sigma_{jk} & = & \left(b_{j}e_{k}-e_{j}b_{k}\right)\left(c_{j}e_{k}-e_{j}c_{k}\right)-\left(b_{j}c_{k}-c_{j}b_{k}\right)^{2}\label{eq:explicit}\\
\Delta_{jkl} & = & \left|\begin{array}{ccc}
b_{j} & b_{k} & b_{l}\\
e_{j} & e_{k} & e_{l}\\
c_{j} & c_{k} & c_{l}\end{array}\right|^{2}.\label{eq:explicitDelta}\end{eqnarray}

Note that $\tau,\sigma$ and $\Delta$ are symmetric under permutation
of any matrices/indices, but $\sigma$ and $\Delta$ vanish when 2
matrices/indices coincide. Since the above expressions do not depend
explicitly on the variables $a_{j}$ or $d_{j}$ but only on the difference
$e_{j}=a_{j}-d_{j}$, the functions $\tau,\sigma$ and $\Delta$ are
invariant under translation of any argument by a scalar matrix, that
is, for any matrices $A,B$ and scalar matrices $\lambda,\mu$, we
have $\tau(A+\lambda,B+\mu)=\tau(A,B)$ and similarly for $\sigma$
and $\Delta$.

\begin{rem}
\label{rem:relations}Note that these are essentially the same functions
used in Drensky's theorem \ref{thm:Drensky} \cite{D}. They also
coincide with the functions used in \cite{Fl}, up to a constant factor.
There are some interesting relations between these invariants which
are obtained from simple calculations. In particular, we have\begin{eqnarray}
\tau(A,A) & = & \delta_{A}=\tr^{2}A-4\de A\nonumber \\
\sigma(A,B) & = & \mathsf{tr}(A[A,B]B)=\frac{1}{4}\left(\tau(A,A)\tau(B,B)-\tau(A,B)^{2}\right)\label{eq:sigma}\\
\Delta(A,B,C) & = & -\frac{1}{4}\left|\begin{array}{ccc}
\tau(A,A) & \tau(A,B) & \tau(A,C)\\
\tau(B,A) & \tau(B,B) & \tau(B,C)\\
\tau(C,A) & \tau(C,B) & \tau(C,C)\end{array}\right|,\nonumber \end{eqnarray}
for all matrices $A,B,C$ over $R$, in agreement with Equation (\ref{eq:relations}).%
{}
\end{rem}
The following is a simple necessary condition for triangularization.

\begin{prop}
\label{pro:sigma=00003D3D0}Let $\mathcal{A}\in V_{n}$ be a triangularizable
sequence. Then $\sigma(A,B)$ and $\Delta(A,B,C)$ vanish for all
terms $A,B,C$ of $\mathcal{A}$. 
\end{prop}
\begin{proof}
Since $\sigma$ and $\Delta$ are $G$ invariant, we can assume that
$\mathcal{A}$ is upper triangular. By direct computation, $\sigma(A,B)=\de([A,B])=0$,
and $\Delta(A,B,C)=(\tr(ABC-CBA))^{2}=0$.
\end{proof}
Note that the vanishing of all $\sigma_{jk}=\sigma(A_{j},A_{k})$
is not sufficient for $\mathcal{A}$ to be triangularizable, as the
following important example shows. We adopt the usual convention that
blank matrix entries stand for zero entries (in this case $0\in R$).

\begin{example}
\label{exa:example}Let $\mathcal{A}=(A_{1},A_{2},A_{3})\in V_{3}$
have the form\[
A_{1}=\left(\begin{array}{cc}
a_{1}\\
 & d_{1}\end{array}\right),\quad A_{2}=\left(\begin{array}{cc}
a_{2} & b_{2}\\
 & d_{2}\end{array}\right),\quad A_{3}=\left(\begin{array}{cc}
a_{3}\\
c_{3} & d_{3}\end{array}\right),\]
for some $a_{1},...,d_{3}\in R.$ Then $\sigma_{12}=\sigma_{13}=0$
and $\sigma_{23}=-b_{2}c_{3}(e_{2}e_{3}+b_{2}c_{3})$. Assume that
$e_{2}e_{3}+b_{2}c_{3}=0$ and that $e_{1}b_{2}c_{3}\neq0$, so that
all $\sigma_{jk}$ vanish, neither $A_{2}$ or $A_{3}$ are diagonal,
and (since these assumptions imply $e_{2}e_{3}\neq0$) all three matrices
have distinct eigenvalues. So, in this case, all subsequences of length
$\leq2$ are triangularizable, but the next Proposition will show
that $\mathcal{A}$ is not triangularizable. 
\end{example}
\begin{prop}
\label{pro:ReducedTriple}As in Example \ref{exa:example}, let $\mathcal{A}=(A_{1},A_{2},A_{3})\in V_{3}$
be a triple of the form\begin{equation}
A_{1}=\left(\begin{array}{cc}
a_{1}\\
 & d_{1}\end{array}\right),\quad A_{2}=\left(\begin{array}{cc}
a_{2} & b_{2}\\
 & d_{2}\end{array}\right),\quad A_{3}=\left(\begin{array}{cc}
a_{3}\\
c_{3} & d_{3}\end{array}\right),\label{eq:form}\end{equation}
with $e_{2}e_{3}\neq0$. Then, the following are equivalent.

(i) $\mathcal{A}$ is reduced (ii) $\mathcal{A}$ is not triangularizable,
(iii) $\Delta_{123}\neq0$ (i.e, $e_{1}b_{2}c_{3}\neq0$). 
\end{prop}
\begin{proof}
(i) or (ii) imply (iii): If $e_{1}b_{2}c_{3}=0$ at least one of the
factors is zero. In each case, $A_{1}$ is a scalar, $\mathcal{A}$
is lower triangular, or $\mathcal{A}$ is upper triangular, respectively,
so $\mathcal{A}$ is triangularizable and is not reduced, since $A_{1}$
commutes with one or both of the other terms. (iii) implies (ii) and
(i): Suppose that $e_{1}b_{2}c_{3}\neq0$. Then, none of the three
numbers is zero. Let $g$ be the $SL_{2}(R)$ matrix with columns
$(x,y)$ and $(z,w)$. Then\begin{eqnarray}
gA_{1}g^{-1} & = & \left(\begin{array}{cc}
* & -xze_{1}\\
ywe_{1} & *\end{array}\right)\nonumber \\
gA_{2}g^{-1} & = & \left(\begin{array}{cc}
* & x(xb_{2}-ze_{2})\\
y(we_{2}-yb_{2}) & *\end{array}\right)\label{eq:BT}\\
gA_{3}g^{-1} & = & \left(\begin{array}{cc}
* & -z(xe_{3}+zc_{3})\\
w(ye_{3}+wc_{3}) & *\end{array}\right),\nonumber \end{eqnarray}
from which it follows that there is no $g\in G$ that will make $g\cdot\mathcal{A}$
upper or lower triangular, so $\mathcal{A}$ is not triangularizable.
Also, by Lemma \ref{lem:comm}, none of the 3 commutators between
the pairs will vanish, so $\mathcal{A}$ is reduced. 
\end{proof}

\subsection{Numerical criteria for triangularization\label{sec:criteria}}

A simple necessary and sufficient numerical condition for triangularization
of a pair of $2\times2$ matrices over an algebraically closed field
was given in the article \cite{Fr}, which also describes the similarity
classes of pairs of $m\times m$ matrices  in great generality.

\begin{prop}
\cite{Fr}\label{pro:Fr} A pair $(A_{1},A_{2})\in V_{2}(\bar{\mathbb{F}})$
is triangularizable if and only if $\sigma_{12}=\de[A_{1},A_{2}]=0$.
\end{prop}
Note that Friedland writes the condition $\sigma_{12}=0$ in a different,
but equivalent form (see Equation (\ref{eq:sigma}) in Remark \ref{rem:relations}).
The generalization to integral domains is as follows.

\begin{thm}
\label{thm: pair}A pair $\mathcal{A}=(A_{1},A_{2})\in V_{2}(R)$
is triangularizable if and only if both $A_{1}$ and $A_{2}$ are
triangularizable and $\sigma_{12}=\de[A_{1},A_{2}]=0$.
\end{thm}
\begin{proof}
The conditions are clearly necessary. For the converse, let us suppose
that both $A_{1}$ and $A_{2}$ are triangularizable and $\de[A_{1},A_{2}]=0$.
Then, as the determinant is $GL_{2}(R)$ invariant, we can assume
$A_{1}$ upper triangular ($c_{1}=0$). If $[A_{1},A_{2}]=0$ the
pair $(A_{1},A_{2})$ is triangularizable by Corollary \ref{cor:comm},
so we can assume that $[A_{1},A_{2}]\neq0$. Equation (\ref{eq:comuta})
shows that \emph{\begin{equation}
0=\de[A_{1},A_{2}]=-b_{1}^{2}c_{2}^{2}+e_{1}c_{2}(e_{1}b_{2}-b_{1}e_{2})=-c_{2}(b_{1}^{2}c_{2}+e_{1}b_{1}e_{2}-e_{1}^{2}b_{2}).\label{eq:up}\end{equation}
}If $c_{2}=0$, $\mathcal{A}$ is triangularizable. If not, $c_{2}\neq0$
and we distinguish 4 cases. (i) If $e_{1}=0$, then $b_{1}\neq0$
(as $A_{1}$ is non-scalar) which makes equation (\ref{eq:up}) impossible
to solve. (ii) If $b_{1}=0$, then $e_{1}\neq0$, and equation (\ref{eq:up})
implies $b_{2}=0$ and $\mathcal{A}$ is lower triangular. (iii) Suppose
now $e_{1}b_{2}=b_{1}e_{2}$. Then $0=\de[A_{1},A_{2}]=-c_{2}^{2}b_{1}^{2}$
and so $b_{1}=0$ which reduces to the previous case. 

Finally, consider the general case (iv) with $\delta_{12}=b_{1}e_{2}-e_{1}b_{2}\neq0$
and non-zero $b_{1}$ and $e_{1}$. So, we are assuming $c_{2}e_{2}b_{2}\neq0$.
From equation (\ref{eq:up}), the quadratic equation $Q_{2}(x,y)\equiv c_{2}x^{2}-e_{2}xy-b_{2}y^{2}=0$
associated to $A_{2}$ (see Appendix A) has a non-trivial solution:
$(b_{1},-e_{1})\in R^{2}$. Suppose that $A_{2}$ is nondegenerate
($\delta_{A}\neq0$). Then, by Lemma \ref{lem:eigenvector}, there
are $z',w'$ in $\mathbb{F}$, the field of fractions of $R$, such
that $w'b_{1}+z'e_{1}\neq0$ and the eigenvectors of $A_{2}$ are
multiples of $(b_{1},-e_{1})$ and of $(z',w')$.%
{} So, we choose an eigenvector of $A_{2}$ of the form $(z'',w'')\in R^{2}$
colinear with $(z',w')\in\mathbb{F}^{2}$. Moreover, by Proposition
\ref{pro:principal}, there are $\alpha,\beta\in\mathbb{F}^{*}=\mathbb{F}\setminus\left\{ 0\right\} $
so that the eigenvectors $(x,y)=\alpha(b_{1},-e_{1})$ and $(z,w)=\beta(z'',w'')$
verify either $xR+yR=R$ or $zR+wR=R$. If the first alternative holds,
let $xq-yp=1$ for some $p,q\in R$, (note that $(p,q)$ is not necessarily
an eigenvector of $A_{2}$) and put:\[
g=\left(\begin{array}{cc}
x & p\\
y & q\end{array}\right).\]
Then, conjugating by $g^{-1}$ gives\begin{eqnarray*}
g^{-1}A_{1}g & = & \left(\begin{array}{cc}
* & *\\
-e_{1}yx-b_{1}y^{2} & *\end{array}\right)=\left(\begin{array}{cc}
* & *\\
\alpha(y^{2}x-xy^{2}) & *\end{array}\right)=\left(\begin{array}{cc}
* & *\\
0 & *\end{array}\right)\\
g^{-1}A_{2}g & = & \left(\begin{array}{cc}
* & *\\
c_{2}x^{2}-e_{2}xy-b_{2}y^{2} & *\end{array}\right)=\left(\begin{array}{cc}
* & *\\
\alpha^{2}Q_{2}(b_{1},-e_{1}) & *\end{array}\right)=\left(\begin{array}{cc}
* & *\\
0 & *\end{array}\right),\end{eqnarray*}
so $\mathcal{A}$ is again triangularizable. If the second alternative
holds, we do the same interchanging the roles of $(x,y)$ and $(z,w)$.
Finally, suppose that $A_{2}$ is degenerate ($\delta_{A}=0$). Then,
there is only one eigenvector of $A_{2}$ and the solutions to $Q_{2}(x,y)=0$
form a single line through the origin in $\mathbb{F}^{2}$, so all
solutions $(x,y)\in R^{2}$ are multiples of $(b_{1},-e_{1})\in R^{2}$.
Since $A_{2}$ is triangularizable, by Proposition \ref{pro:principal},
we can choose $(x,y)$ so that $xR+yR=R$ and we proceed as before.
\end{proof}
\begin{example}
Over the integral domain $R=\mathbb{C}[u,v]$, consider the pair\[
A_{1}=\left(\begin{array}{cc}
-v & u\\
0 & 0\end{array}\right),\quad\quad A_{2}=\left(\begin{array}{cc}
uv & u^{2}\\
2v^{2} & 0\end{array}\right).\]
Then we have $A_{1}$ upper triangular and $\sigma_{12}=0$ as can
be checked, but $A_{2}$ is not triangularizable over $\mathbb{C}[u,v]$,
as no eigenvector $(x,y)\in R^{2}$ satisfies $xR+yR=R$. So, $\sigma_{12}=0$
and $A_{1}$ is  triangularizable but not the pair $(A_{1},A_{2})$.
\end{example}
We finally arrive to Theorem \ref{thm:Flo}, which is a converse to
Proposition \ref{pro:sigma=00003D3D0}.

\begin{thm}
\label{thm:SigmaDelta}A sequence $\mathcal{A}=(A_{1},...,A_{n})\in V_{n}(R)$
is triangularizable if and only if each $A_{j}$ is triangularizable
and $\sigma_{jk}=\Delta_{jkl}=0$ for all $1\leq j,k,l\leq n$.
\end{thm}
\begin{proof}
Consider $\mathcal{A}=(A_{1},...,A_{n})$ reduced with $\sigma_{jk}=\Delta_{jkl}=0$
for all $1\leq j,k,l\leq n$. By Theorem \ref{thm: pair} the conditions
$\sigma_{jk}=0$ and $A_{j}$ triangularizable mean that all subsequences
of $\mathcal{A}$ of length $\leq2$ are triangularizable. So, after
a similarity that puts $A_{1}$ and $A_{2}$ in upper triangular form,
we can assume $(A_{1},A_{2},A_{3})$ to be in the form\[
A_{1}=\left(\begin{array}{cc}
a_{1} & b_{1}\\
 & d_{1}\end{array}\right),\quad A_{2}=\left(\begin{array}{cc}
a_{2} & b_{2}\\
 & d_{2}\end{array}\right),\quad A_{3}=\left(\begin{array}{cc}
a_{3} & b_{3}\\
c_{3} & d_{3}\end{array}\right).\]
Since $\mathcal{A}$ is reduced, by hypothesis $\delta_{12}=b_{1}e_{2}-e_{1}b_{2}\neq0$.
From Equation (\ref{eq:explicitDelta}), we see that $0=\Delta_{123}=(b_{1}e_{2}-e_{1}b_{2})^{2}c_{3}^{2}$,
so $c_{3}=0$ which means that $(A_{1},A_{2},A_{3})$ is triangularizable.
Repeating the argument for all triples $(A_{j},A_{k},A_{l})$ we see
that all subsequences of $\mathcal{A}$ of length $\leq3$ are triangularizable.
So $\mathcal{A}$ is triangularizable by Proposition \ref{pro:leq3}.
Finally, if $\mathcal{A}$ is not reduced, we consider the above argument
for any maximal reduction $\mathcal{B}$, triangularize $\mathcal{B}$,
and apply Corollary \ref{cor:reduce}.
\end{proof}

\subsection{Improved criterion for triangularization}

We now prove an inductive property of reduced sequences that allow
us to improve the criterion of Theorem \ref{thm:SigmaDelta}.

\begin{thm}
\label{thm:induct} Let $n\geq4$ and $\mathcal{A}=(A_{1},...,A_{n-1})\in V_{n-1}$
be a reduced triangularizable sequence. If $\sigma(A_{j},A_{n})=0$
for some matrix $A_{n}$, and all $j=1,...,n-1$ then $(A_{1},...,A_{n})$
is also triangularizable. 
\end{thm}
\begin{proof}
We can suppose that $(A_{1},...,A_{n-1})$ has been conjugated so
that it is already an upper triangular matrix sequence. So all the
$\sigma_{jk}$ vanish, for indices $j,k$ between $1$ and $n-1$
(by Proposition \ref{pro:sigma=00003D3D0}). To reach a contradiction,
assume that $\mathcal{A}'=(A_{1},...,A_{n})$ is not triangularizable
so that \[
A_{n}=\left(\begin{array}{cc}
a_{n} & b_{n}\\
c_{n} & d_{n}\end{array}\right)\]
is non-scalar with $c_{n}\neq0$. Since $\mathcal{A}$ is reduced,
none of the $A_{j}$ can be scalar. We can also assume that $A_{n}$
does not commute with some $A_{j}$ otherwise $\mathcal{A}'$ would
be triangularizable by Proposition \ref{pro:reduce}. This means that
$\delta_{jn}:=b_{j}e_{n}-e_{j}b_{n}\neq0$, for $j=1,...,n-1$, by
Lemma \ref{lem:comm}. Using Formula (\ref{eq:explicit}), the $n-1$
conditions $\sigma_{jn}=0$, $j=1,...,n-1$ can be written as (because
$c_{n}\neq0$)\[
b_{j}^{2}c_{n}+b_{j}e_{j}e_{n}-e_{j}^{2}b_{n}=0,\quad\textrm{for }j=1,...,n-1.\]
Since $A_{n}$ is non-scalar, we are looking for a non-zero solution
$u=(b_{n},e_{n},c_{n})\in R^{3}$ to the matrix equation $Bu=0$ where\[
B=\left(\begin{array}{ccc}
-e_{1}^{2} & e_{1}b_{1} & b_{1}^{2}\\
\vdots & \vdots & \vdots\\
-e_{n-1}^{2} & e_{n-1}b_{n-1} & b_{n-1}^{2}\end{array}\right).\]
A simple computation shows that every $3\times3$ minor of $B$ is
of the form $\pm\delta_{jk}\delta_{kl}\delta_{lj}$. Since all these
minors are non-zero by hypothesis, there is no non-zero solution $u\in R^{3}$,
and we have a contradiction. Hence $c_{n}=0$ and $\mathcal{A}'$
is triangularizable.
\end{proof}
From two finite matrix sequences $\mathcal{A}=(A_{1},...,A_{n_{1}})$
and $\mathcal{B}=(B_{1},...,B_{n_{2}})$ one can form their concatenation
$\mathcal{A}\cup\mathcal{B}:=(A_{1},...,A_{n_{1}},B_{1},...,B_{n_{2}})$.
The following corollary may be called the \emph{concatenation principle}
for triangularizable sequences.

\begin{cor}
If $\mathcal{A}\in V_{n_{1}}$ and $\mathcal{B}\in V_{n_{2}}$ are
triangularizable matrix sequences and they have a common reduction
of length $\geq3$, then their concatenation $\mathcal{A}\cup\mathcal{B}\in V_{n_{1}+n_{2}}$
is also triangularizable. 
\end{cor}
\begin{proof}
Let $\mathcal{C}$ be such a common reduction, which we can assume
to have length 3. $\mathcal{C}=(C_{1},C_{2},C_{3})$ is obviously
triangularizable and $\sigma(C_{j},A_{k})=\sigma(C_{j},B_{k})=0$
for all possible indices, since both $\mathcal{A}$ and $\mathcal{B}$
are triangularizable. So the Proposition above applies.
\end{proof}
To prove Theorems \ref{thm:leq2} and \ref{thm:efic}, we will need
the following easy fact.

\begin{lem}
\label{lem:red-non-ss}Let $n\geq2$. A reduced triangularizable sequence
of length $n$ as at least $n-1$ diagonalizable terms. 
\end{lem}
\begin{proof}
We can assume that $\mathcal{A}$ is already is upper triangular form,
and let $A_{1}$ and $A_{2}$ be two non diagonalizable terms of $\mathcal{A}$.
Then, they are of the form\[
A_{1}=\left(\begin{array}{cc}
a_{1} & b_{1}\\
 & a_{1}\end{array}\right)\quad A_{2}=\left(\begin{array}{cc}
a_{2} & b_{2}\\
 & a_{2}\end{array}\right),\]
for some $a_{1},a_{2},b_{1},b_{2}\in R$. Thus, they commute, so $\mathcal{A}$
is not reduced. 
\end{proof}
\begin{prop}
\label{pro:quadrup}Let $\mathcal{A}=(A_{1},A_{2},A_{3},A_{4})$ be
a reduced quadruple whose terms $A_{j}$ are all triangularizable.
Then $\mathcal{A}$ is triangularizable if and only if $\sigma_{jk}=0$
for all $j,k\in\left\{ 1,2,3,4\right\} $.
\end{prop}
\begin{proof}
One direction is a consequence of Proposition \ref{pro:sigma=00003D3D0}.
Conversely, let $\mathcal{A}$ be reduced with $\sigma_{jk}=0$. By
Theorem \ref{thm: pair} we may assume that $A_{1}$ and $A_{2}$
are upper triangular, so that $c_{1}=c_{2}=0$. From Lemma \ref{lem:red-non-ss},
we can also assume that $A_{1}$ is diagonalizable so that $b_{1}=0$
and $e_{1}\neq0$. Then, reducibility implies that $b_{2}\neq0$.
In order to obtain a contradiction, suppose that $c_{3}c_{4}\neq0$.
Then we have\begin{eqnarray*}
0 & = & \sigma_{13}=e_{1}^{2}b_{3}c_{3}\\
0 & = & \sigma_{14}=e_{1}^{2}b_{4}c_{4}\\
0 & = & \sigma_{23}=-b_{2}c_{3}(b_{2}c_{3}+e_{2}e_{3})\\
0 & = & \sigma_{24}=-c_{4}(b_{2}^{2}c_{4}+b_{2}e_{2}e_{4}-e_{2}^{2}b_{4})\\
0 & = & \sigma_{34}=-b_{4}(c_{3}^{2}b_{4}+c_{3}e_{3}e_{4}-e_{3}^{2}c_{4})\end{eqnarray*}
This implies $b_{3}=0$, $b_{4}=0$, $b_{2}c_{3}+e_{2}e_{3}=0$ and
$b_{2}c_{4}+e_{2}e_{4}=0$. The last two equations imply that $(e_{2},b_{2})$
is a nontrivial solution of the matrix equation\[
\left(\begin{array}{cc}
e_{3} & -c_{3}\\
e_{4} & -c_{4}\end{array}\right)\left(\begin{array}{c}
x\\
y\end{array}\right)=\left(\begin{array}{c}
0\\
0\end{array}\right).\]
This implies $e_{3}c_{4}-c_{3}e_{4}=0$, which together with $b_{3}=b_{4}=0$
contradict the reducibility of $\mathcal{A}$. So, $c_{3}c_{4}=0$
and either $\Delta_{123}=0$ or $\Delta_{124}=0$, by Equation (\ref{eq:explicitDelta}).
Assuming, without loss of generality, that $\Delta_{123}=0$, $(A_{1},A_{2},A_{3})$
is triangularizable by Theorem \ref{thm:SigmaDelta}, and the result
follows from Proposition \ref{pro:ReducedTriple}.
\end{proof}
Now, we are ready to prove Theorems \ref{thm:leq2} and \ref{thm:efic}.

\noindent \emph{Proof of Theorem} \ref{thm:leq2}: Suppose $\mathcal{A}$
is reduced of length $l\geq4$ and assume all subsequences of $\mathcal{A}$
of length $\leq2$ are triangularizable. Then, by Proposition \ref{pro:quadrup}
all quadruples of $\mathcal{A}$ are triangularizable, so that $\mathcal{A}$
is triangularizable by Proposition \ref{pro:leq3}. \hfill{}$\square$

\noindent \emph{Proof of Theorem} \ref{thm:efic}: Let $\mathcal{A}=(A_{1},...,A_{n})$
be a matrix sequence with maximal reduction $\mathcal{A}'=(A_{1},...,A_{l})$
of length $l$. Then $\mathcal{A}$ is triangularizable if and only
$\mathcal{A}'$ is, by Corollary \ref{cor:reduce}, showing (i) for
$l\leq3$. If $l=4$, $(A_{1},A_{2},A_{3},A_{4})$ is triangularizable
by Proposition \ref{pro:quadrup}. Then, we are in the hypothesis
of Theorem \ref{thm:induct}, which shows that we can apply induction
to conclude that $\mathcal{A}'$ is triangularizable for all $l\geq4$.\hfill{}$\square$

\begin{rem}
To summarize and relate our results with McCoy's Theorem, we have
shown that over an algebraically closed field and under the simple
condition that $\mathcal{A}$ has reduced length $\neq3$, $\mathcal{A}$
is triangularizable if and only if the commutators $[A_{j},A_{k}]$
are nilpotent, for $j=1,2,3$ and $k>j$. In particular, the monomials
$p(x)$ in Theorem \ref{thm:McCoy} are unnecessary (for reduced length
$\neq3$).
\end{rem}

\section{Simultaneous Similarity\label{sec:SimSim}}

%
{}

\subsection{Similarity and subtriples}

Again, let $G=GL_{2}(R)$. We now prove Proposition \ref{pro:sim3}
with the help of the following easy lemma. 

\begin{lem}
\label{lem:stabilizer}Let $[A_{1},A_{2}]\neq0$ and $g\in G$. If
$g\cdot(A_{1},A_{2})=(A_{1},A_{2})$ then $g$ is a scalar.
\end{lem}
\begin{proof}
Let $g$ have columns $(x,y)$ and $(z,w)$ with $xw-yz$ invertible.
The conditions $A_{j}g=gA_{j}$, for $j=1,2$ can be written as $M_{j}u=0$,
where \[
M_{j}=\left(\begin{array}{ccc}
0 & c_{j} & -e_{j}\\
-c_{j} & 0 & b_{j}\\
e_{j} & -b_{j} & 0\end{array}\right),\quad\quad u=(z,x-w,y),\quad\quad j=1,2.\]
So, the vector $(z,x-w,y)\in R^{3}$ lies in the intersection of the
nullspace of the $M_{j}$, each being generated by the non-zero vector
$(b_{j},e_{j},c_{j})$, $j=1,2$, (each $M_{j}$ has rank exactly
2, as $A_{1}$ and $A_{2}$ are non-scalars). So, $(z,x-w,y)$ is
zero unless $(b_{1},e_{1},c_{1})$ and $(b_{2},e_{2},c_{2})$ are
colinear. But this is not the case by Equation (\ref{eq:comuta})
since $[A_{1},A_{2}]\neq0$. Thus, $z=y=0$ and $x=w$, so that $g$
is scalar as wanted.
\end{proof}
\noindent \emph{Proof of Proposition} \ref{pro:sim3}: If $\mathcal{A}$
and $\mathcal{B}$ are similar, then $\mathcal{A}_{J}\sim\mathcal{B}_{J}$
for any index set $J=(j_{1},...,j_{l})$ with $l\leq n$ and $1\leq j_{1}<...<j_{l}\leq n$.
Conversely, let $\mathcal{A}_{J}\sim\mathcal{B}_{J}$ for index sets
$J$ of cardinality $\leq3$. We divide the proof into three cases.
(i) Suppose $\mathcal{A}$ is scalar. Since scalar matrices are invariant
under conjugation, $A_{j}\sim B_{j}$ implies that $A_{j}=B_{j}$,
so $\mathcal{A}=\mathcal{B}$. (ii) Suppose now that $\mathcal{A}$
is non-scalar, but commutative. Then some term of $\mathcal{A}$,
say $A_{1}$ is non-scalar. Since $A_{1}\sim B_{1}$, $B_{1}$ is
also non-scalar, and after a conjugation, we can assume that $B_{1}=A_{1}$.
Moreover, $(A_{j},A_{k})\sim(B_{j},B_{k})$ and $[A_{j},A_{k}]=0$
implies that $[B_{j},B_{k}]=0$. As a consequence, $\mathcal{B}$
is also commutative and non-scalar. Since $A_{1}$ is a non-scalar
$2\times2$ matrix, it is a nonderogatory matrix, so that every matrix
commuting with $A_{1}$ is a polynomial in $A_{1}$ with coefficients
in $\mathbb{F}$, the field of fractions of $R$. Therefore $\mathcal{A}=(A_{1},p_{2}(A_{1}),...,p_{n}(A_{1}))$
for some polynomials $p_{j}(x)$, $j=2,...,n$. Since pairs are similar,
let $g_{j}\in G$, $j=2,...,n$, be such that $(B_{1},B_{j})=(A_{1},B_{j})=(g_{j}A_{1}g_{j}^{-1},g_{j}A_{j}g_{j}^{-1})$.
Then $B_{j}=g_{j}A_{j}g_{j}^{-1}=g_{j}\ p_{j}(A_{1})g_{j}^{-1}=p_{j}(g_{j}A_{1}g_{j}^{-1})=p_{j}(A_{1})$,
for all $j=2,...,n$. So $\mathcal{B}=(B_{1},...,B_{n})=(A_{1},p_{2}(A_{1}),...,p_{n}(A_{1}))=\mathcal{A}$.
(iii) Finally, let $\mathcal{A}$ be non-commutative. Then, some pair
does not commute, say $[A_{1},A_{2}]\neq0$. Since all pairs are similar
we may, after a suitable conjugation, assume that $(A_{1},A_{2})=(B_{1},B_{2})$.
Since all triples are similar, let $g_{j}\in G$, $j=3,...,n$, be
such that $g_{j}\cdot(A_{1},A_{2},A_{j})=(A_{1},A_{2},B_{j})$. Then,
since $g_{j}\cdot(A_{1},A_{2})=(A_{1},A_{2})$ Lemma \ref{lem:stabilizer}
implies that $g_{j}$ is scalar for all $j=3,...,n$, so $B_{j}=g_{j}\cdot A_{j}=A_{j}$
and $\mathcal{B}=\mathcal{A}$.\hfill{}$\square$

\subsection{Similarity for semisimple sequences}

We now work over a field $\mathbb{F}$. Recall the following definition.

\begin{defn}
It is called \emph{semisimple} if its $G$-orbit (for the conjugation
action (\ref{eq:accao})) is Zariski closed in $V_{n}$. 
\end{defn}
This notion of semisimplicity generalizes that of a single matrix.
In the context of geometric invariant theory, semisimplicity can be
translated into more algebraic terms as follows. Recall that, in the
general situation of a general affine (algebraic) reductive group
$K$ acting on an affine variety $V$ (both over $\mathbb{F}$) one
defines the affine quotient variety $V\quot K$ as the maximal spectrum
of the ring of invariant functions on $V$, $\Spm\left(\mathbb{F}[V]^{K}\right)$,
which comes equipped with a projection\[
q:V\rightarrow V\quot K\]
induced from the canonical inclusion of algebras $\mathbb{F}[V]^{K}\subset\mathbb{F}[V]$
(see, for example, \cite{MFK} or \cite{Mu}). The set of closed orbits
is in bijective correspondence with geometric points of the quotient
$V\quot K$. Recall also that a vector $x\in V$ is said to be \emph{stable}
if the corresponding `orbit map'\[
\psi_{x}:K\rightarrow V,\quad\quad g\mapsto g\cdot x\]
is proper. It is easy to see that $x\in V$ is stable if and only
if the $K$-orbit of $x$ is closed and its stabilizer $K_{x}$ is
finite. Another useful criterion for stability is the \emph{Hilbert-Mumford
numerical criterion}, which is stated in terms of nontrivial homomorphisms
$\phi:\mathbb{F}^{*}\rightarrow K$, called one parameter subgroups
(1PS) of $K$ ($\mathbb{F}^{*}:=\mathbb{F}\setminus\left\{ 0\right\} $).
To any such $\phi$ and to a point $x\in V$ one associates the composition
$\phi_{x}:=\psi_{x}\circ\phi:\mathbb{F}^{*}\rightarrow V$. If $\phi_{x}$
can be extended to a morphism $\overline{\phi_{x}}:\mathbb{F}\rightarrow V$,
we say that $\lim_{\lambda\to0}\phi_{x}(\lambda)$ exists and equals
$\overline{\phi_{x}}(0)$. 

\begin{thm}
\emph{(Hilbert-Mumford, see \cite{MFK})} \label{thm:HMC}A point
$x\in V$ is \emph{not stable} if and only if there is a one parameter
subgroup $\phi$ of $K$, such that $\overline{\phi_{x}}(0)$ exists.
\end{thm}
%
{}

Let us return to our example of the conjugation action (\ref{eq:accao})
of $G=GL_{2}(\mathbb{F})$ on $V_{n}=V_{n}(\mathbb{F})$, and let
$\mathcal{S}\subset V_{n}=V_{n}(\mathbb{F})$ denote the subset of
semisimple sequences. Then $\mathcal{S}/G$, being the set of closed
orbits, is in bijection with $V_{n}\quot G=\Spm\left(\mathbb{F}[V_{n}]^{G}\right)$.
As described in the introduction, Drensky's theorem (Theorem \ref{thm:Drensky})
realized the algebra of invariants as a quotient $\mathbb{F}[V_{n}]^{G}=\mathbb{F}[\mathbf{x}]/I$
where \[
\mathbf{x}=\left(t_{1},...,t_{n},t_{11},...,t_{nn},u_{12},...,u_{n-1,n},s_{123},...,s_{n-2,n-1,n}\right)\in\mathbb{F}^{N}\]
is the list of generators ($N=2n+\binom{n}{2}+\binom{n}{3}$) and
$I$ the ideal of relations. Dualizing the sequence\[
\mathbb{F}[\mathbf{x}]\to\mathbb{F}[\mathbf{x}]/I=\mathbb{F}[V_{n}]^{G}\subset\mathbb{F}[V_{n}]\]
we obtain:\[
V_{n}\to V_{n}\quot G=\mathcal{S}/G\subset\mathbb{F}^{N}.\]
By standard arguments, the last inclusion is precisely the map $\Phi([\mathcal{A}])=\mathbf{x},$
that sends a $G$ orbit to its values on the generators, so we conclude
the following.

\begin{prop}
\label{pro:phi}For a field of characteristic zero, the map $\Phi:\mathcal{S}/G\to\mathbb{F}^{N}$
is injective.
\end{prop}
To obtain an analogous map for non-semisimple sequences, we first
characterize these ones as triangularizable but not diagonalizable
sequences.

Note that, for the action of $G$ on $V_{n}$ there are no stable
points, since the scalar nonzero matrices stabilize any sequence $\mathcal{A}\in V_{n}$.
This is not a problem, since the same orbit space is obtained with
the conjugation action of the affine reductive group $G_{1}=SL_{2}(\mathbb{F})$
(determinant one matrices in $GL_{2}(\mathbb{F})$) on $V_{n}$ which
has generically $\mathbb{Z}_{2}$ stabilizers:\[
V_{n}\quot G=V_{n}\quot G_{1}.\]
Note that any diagonal matrix sequence $\mathcal{A}\in D:=\left\{ \mathcal{A}\in V_{n}:b=c=0\right\} $
has the subgroup \begin{equation}
H=\left(\begin{array}{cc}
\lambda & 0\\
0 & \lambda^{-1}\end{array}\right)\subset G_{1},\,\,\lambda\in\mathbb{F}^{*}\label{eq:subg}\end{equation}
contained in its stabilizer.

\begin{prop}
\label{Prop:Artin} A $2\times2$ matrix sequence is stable (for the
$SL_{2}(\mathbb{F})$ conjugation action) if and only if it is not
triangularizable. A $2\times2$ matrix sequence is semisimple if and
only if it is either stable or diagonalizable.
\end{prop}
\begin{proof}
This follows from general results (see Artin, \cite{A}). For completeness,
we include a proof of this particular case, using the Hilbert-Mumford
criterion (Theorem \ref{thm:HMC}). For the first statement, we can
assume that $\mathcal{A}$ is upper triangular. Then, a simple computation
shows that the closure of the orbit of $\mathcal{A}$ under the subgroup
$H\subset G_{1}=SL_{2}(\mathbb{F})$ (Equation \ref{eq:subg}) intersects
$D$. So, either $\mathcal{A}$ is in $D$ (and it is commutative
and semisimple) and its stabilizer contains $D$, or $\mathcal{A}$
is not commutative, $\mathcal{A}\notin D$ so $G\cdot\mathcal{A}=G_{1}\cdot\mathcal{A}$
is not closed. In either case, $\mathcal{A}$ is not stable. Conversely,
let $\mathcal{A}$ be not stable for the action of $G_{1}$. By elementary
representation theory, any one parameter subgroup of $G_{1}$ is conjugated
to\[
\lambda\mapsto\phi_{n}(\lambda)=\left(\begin{array}{cc}
\lambda^{n} & 0\\
0 & \lambda^{-n}\end{array}\right),\quad n\in\left\{ 1,2,...\right\} .\]
In other words, any 1PS can be written as $\phi=g^{-1}\phi_{n}g$,
for some $g\in G_{1}$ and some $\phi_{n}$ so,\begin{equation}
\lim_{\lambda\rightarrow0}\phi_{\mathcal{A}}(\lambda)=\lim_{\lambda\rightarrow0}\phi(\lambda)\cdot\mathcal{A}=g^{-1}\lim_{\lambda\rightarrow0}\phi_{\mathsf{n}}(\lambda)\cdot(g\cdot\mathcal{A}).\label{eq:limit}\end{equation}
 Writing $g\cdot\mathcal{A}$ as\[
g\cdot\mathcal{A}=\left(\begin{array}{cc}
a(g) & b(g)\\
c(g) & d(g)\end{array}\right),\]
 we obtain \[
\phi_{n}(\lambda)\cdot(g\cdot\mathcal{A})=\left(\begin{array}{ll}
a(g) & b(g)\lambda^{2n}\\
c(g)\lambda^{-2n} & d(g)\end{array}\right).\]
By the Hilbert-Mumford criterion, the limit (\ref{eq:limit}) exists
for some 1PS, so we must have $c(g)=0$, for some $g\in G_{1}$. This
means that $g\cdot\mathcal{A}$ is upper triangular, hence not irreducible.
The second statement is analogous.
\end{proof}

\subsection{Similarity for non-commutative triangularizable sequences}

To obtain a numerical similarity criterion for non-semisimple sequences,
we are reduced, by Proposition \ref{Prop:Artin}, to the case of triangularizable
sequences. Here, we consider the non-commutative triangularizable
case.

Let $\mathsf{U}\subset V_{n}$ be the affine variety of upper triangular
matrix sequences and let $\mathsf{K}\subset\mathsf{U}$ be the subset
of commutative sequences. For $n\geq2$ and $\mathcal{A}\in\mathsf{U}$,
let $P_{\mathcal{A}}$ denote the following $2\times n$ matrix, and
$\delta_{jk}=\delta_{jk}(\mathcal{A})$ the corresponding $2\times2$
minors \[
P_{\mathcal{A}}=\left(\begin{array}{ccc}
e_{1} & \cdots & e_{n}\\
b_{1} & \cdots & b_{n}\end{array}\right),\quad\quad\delta_{jk}=b_{j}e_{k}-e_{j}b_{k},\quad\quad j,k\in\{1,...,n\}.\]
According to Lemma \ref{lem:comm}, $\mathsf{K}$ is characterized
as consisting of sequences $\mathcal{A}$ such that the rank of $P_{\mathcal{A}}$
is $\leq1$ (it is $0$ when all terms in $\mathcal{A}$ are scalars).
Hence, when $\mathcal{A}\in\mathsf{U}\setminus\mathsf{K}$, $P_{\mathcal{A}}$
has rank $2$ and defines an element $\pi_{\mathcal{A}}$ of the Grassmanian
$\mathbb{G}(2,n)$ of 2-planes in $\mathbb{F}^{n}$. It is easy to
see that the action of $G=GL_{2}(\mathbb{F})$, restricted to $\mathsf{U}\setminus\mathsf{K}$,
preserves the plane $\pi_{\mathcal{A}}$:

\begin{lem}
\label{lem:e=00003De'}Let $g\in G$. If $\mathcal{A}$ and $\mathcal{A}'=g\cdot\mathcal{A}$
are both in $\mathsf{U}\setminus\mathsf{K}$, then $e=e'$ and $\pi_{\mathcal{A}}=\pi_{\mathcal{A}'}$. 
\end{lem}
\begin{proof}
We can assume that $\mathcal{A}'=g\cdot\mathcal{A}$ for some $g\in SL_{2}(\mathbb{F})$
and compute, for $\mathcal{A}\in\mathsf{U}\setminus\mathsf{K}$, \[
g\cdot\mathcal{A}=\left(\begin{array}{cc}
x & z\\
y & w\end{array}\right)\left(\begin{array}{cc}
a & b\\
0 & d\end{array}\right)\left(\begin{array}{cc}
w & -z\\
-y & x\end{array}\right)=\left(\begin{array}{cc}
* & *\\
y(ew-by) & *\end{array}\right).\]
Since $P_{\mathcal{A}}$ has rank $2$, $ew\neq by$ as vectors in
$\mathbb{F}^{n}$. Thus, in order that $g\cdot\mathcal{A}$ be again
in $\mathsf{U}\setminus\mathsf{K}$, we need $y=0$, which simplifies
the formula above to\begin{equation}
g\cdot\mathcal{A}=\left(\begin{array}{cc}
x & z\\
0 & w\end{array}\right)\left(\begin{array}{cc}
a & b\\
0 & d\end{array}\right)\left(\begin{array}{cc}
w & -z\\
0 & x\end{array}\right)=\left(\begin{array}{cc}
a & x(bx-ez)\\
0 & d\end{array}\right),\label{eq:plane}\end{equation}
because $xw=\de g=1$. Then, $e=e'=a-d$. Moreover, the upper right
entry is $b'=x(bx-ez)=x^{2}b-xze$, a linear combination of $b$ and
$e$. So $\pi_{\mathcal{A}}=\pi_{\mathcal{A}'}$, as asserted. 
\end{proof}
The correspondence $\mathcal{A}\mapsto\pi_{\mathcal{A}}$ is therefore
well defined on the $G$-orbits of non-commutative triangularizable
sequences. To make this more precise, consider the following algebraic
subvarieties of $V_{n}$. Let $\mathcal{U}=G\cdot\mathsf{U}$ be the
variety of triangularizable matrix sequences and let $\mathcal{K}\subset\mathcal{U}$
be the subvariety of commutative sequences. Note that $\mathcal{U}$
is indeed irreducible, as the image in $V_{n}$ of the (irreducible)
algebraic variety $G\times\mathsf{U}$ under the morphism $(g,\mathcal{A})\mapsto g\cdot\mathcal{A}$.
The irreducibility of $\mathcal{K}$ was first noted in \cite{Ge}
(see \cite{Gu} for a proof).

By the previous lemma, given a sequence $\mathcal{A}\in\mathcal{U}\setminus\mathcal{K}$,
we can define $\pi_{\mathcal{A}}:=\pi_{\mathcal{B}}$ using any matrix
sequence $\mathcal{B}\in G\cdot\mathcal{A}\cap\mathsf{U}$, that is,
$\mathcal{B}$ is upper triangular and similar to $\mathcal{A}$ (even
though the matrix $P_{\mathcal{A}}$ is not defined). So, we can define
a map $\psi:\mathcal{U}\backslash\mathcal{K}\to\mathbb{P}^{N-1}$,
where $\mathbb{P}^{N-1}$ denotes the projective space over $\mathbb{F}$
of dimension $N-1=\binom{n}{2}-1$, as the composition\begin{eqnarray*}
\mathcal{U}\backslash\mathcal{K}\\
\pi\downarrow & \searrow\psi\\
\mathbb{G}(2,n) & \hookrightarrow & \mathbb{P}^{N-1}\end{eqnarray*}
where the bottom inclusion is the Plücker embedding of the Grassmanian
$\mathbb{G}(2,n)$. In more concrete terms we have.

\begin{lem}
\label{lem:Plucker}Let $\mathcal{A}\in\mathcal{U}\backslash\mathcal{K}$.
Then $\psi(\mathcal{A})=[\delta_{12}:\delta_{13}:\cdots:\delta_{n-1,n}]$,
where $\delta_{jk}=\delta_{jk}(\mathcal{B})$, for any $\mathcal{B}\in G\cdot\mathcal{A}\cap\mathsf{U}$. 
\end{lem}
\begin{proof}
The formula follows from the definition of the Plücker embedding.
We just need to show that the map is well defined. The point $[\delta_{12}:\delta_{13}:\cdots:\delta_{n-1,n}]$
is in projective space since at least a pair of terms, say $A_{j}$
and $A_{k}$ do not commute, so that $\delta_{jk}(\mathcal{B})=b_{j}e_{k}-e_{j}b_{k}$
is nonzero, by Lemma \ref{lem:comm}. On the other hand for a different
sequence $\mathcal{B}'\in G\cdot\mathcal{A}\cap\mathsf{U}$, Equation
(\ref{eq:plane}) and $e=e'$ (Lemma \ref{lem:e=00003De'}) imply
\[
\delta_{jk}(\mathcal{B}')=b'_{j}e'_{k}-e'_{j}b'_{k}=x(b_{j}x-e_{j}z)e_{k}-e_{j}x(b_{k}x-e_{k}z)=x^{2}\delta_{jk}(\mathcal{B}),\]
so the point $\psi(\mathcal{A})\in\mathbb{P}^{N-1}$ is indeed independent
of the choice of $\mathcal{B}\in G\cdot\mathcal{A}\cap\mathsf{U}$. 
\end{proof}
%
{}By Lemma \ref{lem:Plucker}, the quotients $\delta_{jk}/\delta_{lm}$
(for $j\neq k$ and $l\neq m$) are well defined rational $G$-invariant
functions on $\mathcal{U}\setminus\mathcal{K}$ (defined on the open
dense complement of $\delta_{lm}^{-1}(0)\subset\mathcal{U}\setminus\mathcal{K}$),
so they descend to the quotient space $(\mathcal{U}\setminus\mathcal{K})/G$.
Note, however, that these are not quotients of regular (polynomial)
invariants.%
{}

Define now the map $\Psi:(\mathcal{U}\setminus\mathcal{K})/G\to\mathbb{F}^{n}\times\mathbb{F}^{N+n}\times\mathbb{P}^{N-1}$,
$N=\binom{n}{2}$, by \[
\Psi([\mathcal{A}])=\left(\{t_{j}\}_{j=1,...,n},\ \{t_{jk}\}_{j,k=1,...,n},\ \psi\right),\]
where $[\mathcal{A}]$ denotes the conjugacy class of $\mathcal{A}$,
$t_{j}=\tr(A_{j})$, $t_{jk}=\tr(A_{j}A_{k})$, and $\psi$ was given
by Lemma \ref{lem:Plucker}. The main result of this subsection is
the following. For the proof, we use a standard consequence of the
Noether-Deuring theorem; namely, if $\mathcal{A}$ and $\mathcal{A}'$
are elements in $V_{m,n}(\mathbb{F})$ and $g\in GL_{m}(\bar{\mathbb{F}})$
verifies $g\cdot\mathcal{A}=\mathcal{A}'$ (i.e, similarity over $\bar{\mathbb{F}}$)
then they are similar over $\mathbb{F}$ (see eg. \cite{CR} p. 200).

\begin{thm}
\label{thm:rational-invariants}Let $n\geq2$. The map $\Psi$ is
two-to-one. More precisely, the rational invariants $\delta_{jk}/\delta_{lm}$,
together with the regular invariants $t_{j}$ and $t_{jk}$, distinguish
$G$-orbits in $\mathcal{U}\setminus\mathcal{K}$, except for the
identification $e\leftrightarrow-e$, when written in triangular form.
\end{thm}
\begin{proof}
Let both sequences $\mathcal{A}$ and $\mathcal{A}'$ be triangularizable
and non-commutative, so there exist $\mathcal{B}\in G\cdot\mathcal{A}\cap(\mathsf{U}\setminus\mathsf{K})$
and $\mathcal{B}'\in G\cdot\mathcal{A}'\cap(\mathsf{U}\setminus\mathsf{K})$.
Let us write $t_{j}=\tr(B_{j})=\tr(A_{j})$ etc, as before, and use
primed letters to denote corresponding quantities for $\mathcal{A}'$
or $\mathcal{B}'$. To prove injectivity, suppose $\Psi([\mathcal{A}])=\Psi([\mathcal{A}'])$
so that $t_{j}=t_{j}'$ and $t_{jk}=t_{jk}'$ for all appropriate
indices. Then, \[
e_{j}e_{k}=2t_{jk}-t_{j}t_{k}=2t_{jk}'-t_{j}'t_{k}'=e_{j}'e_{k}'.\]
Since any triple $(x^{2},xy,y^{2})\in\mathbb{F}^{3}$ determines $(x,y)$
up to sign, the equation above implies that $e'=\pm e$, as vectors
in $\mathbb{F}^{n}$. Suppose $e'=e$. The hypothesis $\Psi([\mathcal{A}])=\Psi([\mathcal{A}'])$
means also that $\psi(\mathcal{B})=\psi(\mathcal{B}')$. Thus $\pi_{\mathcal{B}}=\pi_{\mathcal{B}'}$
as planes in $\mathbb{G}(2,n)$ due to the injectivity of the Plücker
map. Therefore $b'=\alpha b+\beta e$ for some complex numbers $\alpha,\beta$,
$\alpha\neq0$. So, using the invertible \[
g=\left(\begin{array}{cc}
\sqrt{\alpha} & -\beta/\sqrt{\alpha}\\
 & \sqrt{\alpha}^{-1}\end{array}\right),\]
where $\sqrt{\alpha}$ is any square root of $\alpha$ in $\bar{\mathbb{F}}$,
it is easy to see that $g\cdot\mathcal{B}=\mathcal{B}'$, so that
$\mathcal{A}$ and $\mathcal{A}'$ are similar over $\bar{\mathbb{F}}$.
So, $\mathcal{A}$ and $\mathcal{A}'$ are also similar over $\mathbb{F}$.
Finally, with $e'=-e$ (note $e\neq0$ by hypothesis), one can see
that $\mathcal{B}'$ is not similar to $\mathcal{B}$, hence the result.%
{}
\end{proof}
In view of Proposition \ref{Prop:Artin}, Theorem \ref{thm:rational-invariants}
and Proposition \ref{pro:phi} give a solution to the invariants problem
(ii) for non-commutative sequences. Theorem \ref{thm:Invariants}
provides a more efficient solution (and works for any characteristic
$\neq2$), and its proof will be given after exploring the canonical
forms described in the next section.

\begin{rem}
Finding an analogous map in the commutative case is more involved!
See Friedland, \cite{Fr} for a discussion of the case of a commuting
pair of $2\times2$ matrices. However, testing similarity of commutative
$2\times2$ matrix sequences is a trivial task after reduction to
triangular form, as recalled in Appendix B.
\end{rem}
%
{}We end this section with the following conjecture on the generalization
of Proposition \ref{pro:sim3} to $m\times m$ matrices. Since an
irreducible sequence is a generic sequence and the conjugacy classes
of triangularizable or block-triangularizable sequences seem to depend
on less data (i.e, less regular or rational invariants) than the irreducible
case, we propose the following problem. Define the {}``semisimple
similarity number'' $S(m)$ as\[
S(m)=\min\{k\in\mathbb{N}:\quad\forall n\in\mathbb{N},\ \forall\mathcal{A},\mathcal{B}\in\mathcal{S}_{n,m},\quad\mathcal{A}\sim\mathcal{B}\Leftrightarrow\mathcal{A}_{J}\sim\mathcal{B}_{J}\ \forall J\mbox{ with }|J|\leq k\},\]
where $\mathcal{S}_{n,m}\subset V_{n,m}$ is the subset of semisimple
sequences. 

\begin{conjecture*}
For arbitrary $m\times m$ sequences (not necessarily semisimple)
$\mathcal{A}\sim\mathcal{B}$ if and only if $\mathcal{A}_{J}\sim\mathcal{B}_{J}$
for all index vectors $J$ with length $\leq S(m)$.
\end{conjecture*}
Note that, by Procesi's Theorem, $S(m)\leq N(m)$. This conjecture
is true for $m=2$, by Proposition \ref{pro:sim3}, since Proposition
\ref{pro:ReducedTriple} gives $S(2)=3$.

\section{Canonical Forms and Reconstruction of Sequences\label{sec:CanonicalForms}}

\subsection{Canonical forms}

We now describe canonical forms for $2\times2$ matrix sequences over
an algebraically closed field $\bar{\mathbb{F}}$. Informally, this
means the indication, for each given matrix sequence $\mathcal{A}\in V_{n}$,
of an element in its conjugacy class which has a simple form, and
this same simple form should be used for the biggest possible set
of sequences, hence the term `canonical'. We will assume $n\geq2$
since such canonical forms for $n=1$ are provided by the well known
Jordan decomposition. Also, for commutative sequences, the canonical
forms are the same as for $n=1$. This is recalled in Appendix B.
For non-commutative sequences, it turns out that canonical forms can
be divided into 5 cases.%
{}

By the preceding results, it is no surprise that we need to consider
distinct canonical forms for the stable and for the reducible cases.
We make the following choices.

\begin{defn}
We say that a stable (i.e, irreducible) matrix sequence $\mathcal{A}=(A_{1},A_{2},...)$
is in \emph{(stable) canonical form} if $A_{1}$ is in Jordan canonical
form and $b_{2}=1$. We say that a triangularizable sequence $\mathcal{A}$
is in \emph{(triangular) canonical form} if $\mathcal{A}$ is upper
triangular, $A_{1}$ is diagonal and $b_{2}=1$.
\end{defn}
Recall that, by Proposition \ref{lem:red-non-ss} a triangularizable
sequence of length $n\geq2$ has at least $n-1$ semisimple (or diagonalizable)
terms. By contrast, a stable sequence can have all terms which are
non-diagonalizable. To see this, just consider the family of matrices
of the form\[
\left(\begin{array}{cc}
-\alpha\beta & \alpha^{2}\\
-\beta^{2} & \alpha\beta\end{array}\right),\]
for some parameters $\alpha$ and $\beta$ not both zero. Then, all
these matrices are similar to the $2\times2$ Jordan block with zero
diagonal (i.e, $\alpha=1$, $\beta=0$), but the $\sigma$ of two
of these is zero only when the vectors $(\alpha,\beta)$ are colinear.
We have, however, the following fact.

\begin{prop}
\label{pro:ss->sigma}Let $\mathcal{A}$ have reduced length $n=2$
or $n\geq4$. Then $\mathcal{A}$ is stable if and only if some $\sigma_{jk}\neq0$. 
\end{prop}
\begin{proof}
If $\mathcal{A}$ is not stable, then it is triangularizable, by Proposition
\ref{Prop:Artin}, so that all $\sigma=0$, by Proposition \ref{pro:sigma=00003D3D0}.
Conversely, suppose $\mathcal{A}$ is stable with reduced length $n=2$.
Then $\sigma_{12}\neq0$ because of Friedland's result (or Theorem
\ref{thm:Flo}). Finally, if $n\geq4$, Proposition \ref{pro:quadrup}
implies that at least one $\sigma_{jk}\neq0$. On the other hand,
Example \ref{exa:example} shows that the statement is not true for
$n=3$. 
\end{proof}
%
{}

\begin{thm}
\label{thm:canonical}All non-commutative matrix sequences can be
put in canonical form. More precisely, after rearranging terms, any
sequence is similar to a sequence with $A_{1},A_{2}$ and $A_{3}$
as described by the following table. 
\end{thm}
\begin{center}
\renewcommand{\multirowsetup}{\centering}\begin{tabular}{|c|c|c|c|}
\hline 
\multirow{3}{35mm}{stable case} & 1.a & $\sigma_{12}\neq0$, $\delta_{1}\neq0$  & $A_{1}$ diagonal, $b_{2}=1$\tabularnewline
\cline{2-2} \cline{3-3} \cline{4-4} 
\multicolumn{1}{|c|}{} & 1.b & $\sigma_{12}\neq0$, $\delta_{1}=\delta_{2}=0$ & $A_{1}$ Jordan block, $b_{2}=1$\tabularnewline
\cline{2-2} \cline{3-3} \cline{4-4} 
\multicolumn{1}{|c|}{} & 1.c & $\sigma_{12}=\sigma_{13}=\sigma_{23}=0$, $\Delta_{123}\neq0$ & $A_{1}$ diagonal, $b_{2}=1$\tabularnewline
\hline 
\multirow{2}{35mm}{triangularizable case} & 2.a & all diagonalizable  & $A_{1}$ diagonal, $b_{2}=1$ \tabularnewline
\cline{2-2} \cline{3-3} \cline{4-4} 
\multicolumn{1}{|c|}{} & 2.b & 1 non-diagonalizable ($A_{2}$) & $A_{1}$ diag., $A_{2}$ Jordan bl.\tabularnewline
\hline
\end{tabular}\renewcommand{\multirowsetup}{\raggedright}
\par\end{center}

\begin{proof}
Let us show that only the five possibilities above occur, and may
be put in the given forms, starting when $\mathcal{A}$ is stable.
If $n\geq4$ or $n=2$, Proposition \ref{pro:ss->sigma} implies that
there is some $\sigma\neq0$, so we can rearrange the terms of $\mathcal{A}$
so that we have $\sigma_{12}\neq0$ (in particular, $[A_{1},A_{2}]\neq0$).
If $A_{1}$ or $A_{2}$ is diagonalizable, we have possibility 1.a.
Assuming $A_{1}$ diagonalizable, we may suppose that $A_{1}$ is
already in diagonal form. Since $A_{1}$ is nonscalar $\delta_{1}\neq0$,
and the stabilizer of $A_{1}$ are the diagonal invertible matrices
$H\subset G$. Then $\sigma_{12}=e_{1}^{2}b_{2}c_{2}\neq0$ implies
that both $b_{2}$ and $c_{2}$ are nonzero. Let $g=\mbox{diag}(x,x^{-1})$
for some $x\in\bar{\mathbb{F}}^{*}$. A simple computation shows that
$g^{-1}A_{1}g$ is diagonal and $A_{2}'=g^{-1}A_{2}g$ has $b_{2}'=x^{-2}b_{2}$.
So, we can solve for $x$ in order to have $b_{2}'=1$, as claimed.
If $A_{1}$ and $A_{2}$ are both not diagonalizable, we have case
1.b. In this case, we may suppose $A_{1}$ is already in Jordan form.
Conjugating $A_{2}$ with a matrix of the form \[
g=\left(\begin{array}{cc}
1 & z\\
 & 1\end{array}\right),\]
a simple computation shows that we can solve for $z$ in order to
obtain $b_{2}=1$. If $n=3$ and some $\sigma_{12}\neq0$, we return
to cases 1.a or 1.b. So it remains the case $n=3$ and all $\sigma=0$,
which is 1.c. From Example \ref{exa:example} we see that necessarily
$\Delta_{123}\neq0$ (after an eventual rearrangement of terms), and
a diagonal conjugation will achieve $b_{2}=1$. Finally, when $\mathcal{A}$
is triangularizable, so by Proposition \ref{lem:red-non-ss} either
none or one of the terms of $\mathcal{A}$ are non-semisimple, cases
2.a and 2.b respectively. The forms mentioned will be obtained by
conjugation with a diagonal invertible matrix.
\end{proof}
In the above table we have $[A_{1},A_{2}]\neq0$ in all cases, so
if $g$ stabilizes $(A_{1},A_{2})$, then $g$ is scalar, by Lemma
\ref{lem:stabilizer}. Therefore, we have a uniqueness statement.

\begin{prop}
\label{pro:UniqueCanonical}Let $\mathcal{A},\mathcal{B}$ be two
sequences in canonical form with $[A_{1},A_{2}]\neq0$. Assume that
$(B_{1},B_{2})=(A_{1},A_{2})$. Then, $\mathcal{A}\sim\mathcal{B}$
if and only if $\mathcal{A}=\mathcal{B}$.
\end{prop}
%
{}

\subsection{Reconstruction of sequences from invariants}

We continue to work over an algebraically closed field $\bar{\mathbb{F}}$,
and here we restrict to characteristic $\neq2$. Let $\mathcal{S}'(\bar{\mathbb{F}})$
denote the subset of $V_{n}(\bar{\mathbb{F}})$ of semisimple sequences
such that $A_{1}$ is diagonalizable and $[A_{1},A_{2}]\neq0$. Then
$\mathcal{S}'(\mathbb{F})$ consists of irreducible (stable) sequences
that can be put in the form 1.a.%
{} Let $\bar{\Phi}:\mathcal{S}'(\bar{\mathbb{F}})/G(\bar{\mathbb{F}})\to\bar{\mathbb{F}}^{4n-3}$
denote the map\begin{equation}
\bar{\Phi}([\mathcal{A}]):=\left(t_{1},t_{11},t_{2},t_{22},t_{12},...,t_{k},t_{1k},t_{2k},s_{12k}...,t_{n},t_{1n},t_{2n},s_{12n}\right).\label{eq:PhiBar}\end{equation}

We now describe a process of (re)constructing a sequence from its
values under the map $\bar{\Phi}$. Let $v\in\bar{\mathbb{F}}^{4n-3}$
be given. We want to find $\mathcal{A}\in\mathcal{S}'(\bar{\mathbb{F}})$
such that $\bar{\Phi}([\mathcal{A}])=v$. 

Let $a_{1}$ and $d_{1}$ (in $\bar{\mathbb{F}}$) be the roots of
the polynomial $\lambda^{2}-t_{1}\lambda+\frac{t_{1}^{2}-t_{11}}{2}=0$
(the characteristic polynomial of a diagonal matrix whose trace is
$t_{1}$ and the trace of its square is $t_{11}$). If they are equal,
there is no solution to our problem simply because an $\mathcal{A}$
satisfying $\bar{\Phi}([\mathcal{A}])=v$ will have either $A_{1}$
non-diagonalizable or $[A_{1},A_{2}]=0$. So, with $e_{1}=a_{1}-b_{1}\neq0$,
put $b_{1}=c_{1}=0$, $b_{2}=1$ and\begin{eqnarray}
c_{2} & = & \frac{t_{22}-a_{2}^{2}-d_{2}^{2}}{2}=\frac{1}{4e_{1}^{2}}\left|\begin{array}{cc}
\tau_{11} & \tau_{12}\\
\tau_{21} & \tau_{22}\end{array}\right|\neq0,\nonumber \\
\left(\begin{array}{c}
a_{2}\\
d_{2}\end{array}\right) & = & \left(\begin{array}{cc}
1 & 1\\
a_{1} & d_{1}\end{array}\right)^{-1}\left(\begin{array}{c}
t_{2}\\
t_{12}\end{array}\right).\label{eq:ReconstructPair}\end{eqnarray}
Then, one easily checks that the pair $(A_{1},A_{2})\in V_{2}(\bar{\mathbb{F}})$
whose entries are $a,b,c,d\in\bar{\mathbb{F}}^{2}$ is in canonical
form 1.a and satisfies $\bar{\Phi}([\mathcal{A}])=(t_{1},t_{11},t_{2},t_{22},t_{12})$.
Moreover, this pair is unique except for the choice of assigning to
$a_{1}$ or $d_{1}$ one or the other root of the characteristic polynomial.
The two possible choices are:\[
\left(\left(\begin{array}{cc}
a_{1}\\
 & d_{1}\end{array}\right),\left(\begin{array}{cc}
a_{2} & 1\\
c_{2} & d_{2}\end{array}\right)\right),\quad\left(\left(\begin{array}{cc}
d_{1}\\
 & a_{1}\end{array}\right),\left(\begin{array}{cc}
d_{2} & 1\\
c_{2} & a_{2}\end{array}\right)\right),\]
 These pairs are similar (for $e_{1}$ and $c_{2}$ both non-zero),
with similarity matrix\begin{equation}
g=\left(\begin{array}{cc}
 & 1/x\\
-x\end{array}\right),\quad\quad x=\sqrt{-c_{2}}.\label{eq:a<->d}\end{equation}
Moreover, up to a non-zero scalar multiple, this is the unique matrix
sending one pair to the other. Let $\mathcal{A}^{\vee}:=g\cdot\mathcal{A}$
denote the sequence obtained by acting with this matrix. Note that
$[\mathcal{A}]=[\mathcal{A}^{\vee}]$ and if $\mathcal{A}$ is in
canonical form, then so is $\mathcal{A}^{\vee}$. Now, let\begin{equation}
\left(\begin{array}{c}
a_{k}\\
b_{k}\\
c_{k}\\
d_{k}\end{array}\right)=\left(\begin{array}{cccc}
1 & 0 & 0 & 1\\
a_{1} & 0 & 0 & d_{1}\\
a_{2} & c_{2} & 1 & d_{2}\\
0 & -c_{2}e_{1} & e_{1} & 0\end{array}\right)^{-1}\left(\begin{array}{c}
t_{k}\\
t_{1k}\\
t_{2k}\\
s_{12k}\end{array}\right),\quad\quad k=3,...,n.\label{eq:ReconstructSS}\end{equation}
The determinant of this matrix is $-2e_{1}^{2}c_{2}$, so our hypothesis
imply that it is indeed invertible. Hence the transformation above
provides an isomorphism of vector spaces, from the variables $(a_{k},b_{k},c_{k},d_{k})$
to the variables $(t_{k},t_{1k},t_{2k},s_{12k})$, $k=3,...,n$. 

Now, consider the reconstruction of triangularizable sequences. Let
$\mathcal{U}'(\bar{\mathbb{F}})$ denote the subset of $V_{n}(\bar{\mathbb{F}})$
of triangularizable sequences such that $A_{1}$ is diagonalizable
and $[A_{1},A_{2}]\neq0$ (so that $\mathcal{U}'(\bar{\mathbb{F}})\subset\mathcal{U}\setminus\mathcal{K}$).
Define the map $\bar{\Psi}:\mathcal{U}'(\bar{\mathbb{F}})/G\to\bar{\mathbb{F}}^{2n}\times\mathbb{P}^{n-2}$,
where $\mathbb{P}^{k}$ denotes the projective space over $\bar{\mathbb{F}}$
of dimension $k$, by the formula \begin{eqnarray*}
\bar{\Psi}([\mathcal{A}]) & = & \left(t_{1},t_{11},...,t_{k},t_{1k},...,t_{n},t_{1n};\ \psi'\right),\end{eqnarray*}
where $\psi'([\mathcal{A}])=[1:\delta_{13}:\cdots:\delta_{1n}]\in\mathbb{P}^{n-2}$,
$\delta_{jk}:=b_{j}e_{k}-e_{j}b_{k}$ assuming $\mathcal{A}$ is in
upper triangular form. Let $w\in\bar{\mathbb{F}}^{2n}\times\mathbb{P}^{n-2}$
be given. Again, let $a_{1}$ and $d_{1}$ be the (distinct) roots
of the characteristic polynomial $\lambda^{2}-t_{1}\lambda+\frac{t_{1}^{2}-t_{11}}{2}=0$.
Then, with $e_{1}=a_{1}-d_{1}\neq0$, put $b_{1}=0$, $b_{2}=1$,
$c_{j}=0$, $j=1,...,n$, and\begin{eqnarray*}
\left(\begin{array}{c}
a_{k}\\
d_{k}\end{array}\right) & = & \left(\begin{array}{cc}
1 & 1\\
a_{1} & d_{1}\end{array}\right)^{-1}\left(\begin{array}{c}
t_{k}\\
t_{1k}\end{array}\right),\\
b_{k} & = & -\frac{\delta_{1k}}{e_{1}},\quad k=3,...,n.\end{eqnarray*}
Then the matrix sequence $\mathcal{A}$ with the entries $a,b,c,d\in\bar{\mathbb{F}}$
is in triangular canonical form (2.a or 2.b) and satisfies $\bar{\Psi}([\mathcal{A}])=w$,
with $\psi'(\mathcal{A})=[1:\delta_{13}:\cdots:\delta_{1n}]\in\mathbb{P}^{n-2}$.
The other solution $\mathcal{A}'$ is obtained by changing $e$ to
$-e$ (that is, exchanging $a$ with $d$). We have.

\begin{prop}
The map $\bar{\Phi}:\mathcal{S}'(\bar{\mathbb{F}})/G(\bar{\mathbb{F}})\to\bar{\mathbb{F}}^{4n-3}$
is injective, and $\bar{\Psi}:\mathcal{U}'(\bar{\mathbb{F}})/G(\bar{\mathbb{F}})\to\bar{\mathbb{F}}^{2n}\times\mathbb{P}^{n-2}$
is two-to-one.
\end{prop}
\begin{proof}
Let $\mathcal{A}_{c}=(A_{1},A_{2},...)$ and $\mathcal{B}_{c}=(B_{1},B_{2},...)$
be canonical forms associated to $\mathcal{A},\mathcal{B}\in\mathcal{S}'(\bar{\mathbb{F}})$
respectively. If $\bar{\Phi}([\mathcal{A}_{c}])=\bar{\Phi}([\mathcal{B}_{c}])$
then either $(A_{1},A_{2})=(B_{1},B_{2})$ or $(A_{1},A_{2})=(B_{1},B_{2})^{\vee}$.
If the first situation holds, then $\mathcal{A}_{c}=\mathcal{B}_{c}$
by the isomorphism given by Equation (\ref{eq:ReconstructSS}). Hence,
Proposition \ref{pro:UniqueCanonical} shows that $\mathcal{A}\sim\mathcal{B}$.
In the second alternative, just replace $\mathcal{B}_{c}$ with $\mathcal{B}_{c}^{\vee}$
(using $g$ in Equation (\ref{eq:a<->d})) and the conclusion is the
same, showing that $\bar{\Phi}$ is injective. The case of $\bar{\Psi}$
is analogous, although in this case the two possibilities for a pair
in triangular canonical form with given values of $(t_{1},t_{11},t_{2},t_{22})$
are:\[
\left(\left(\begin{array}{cc}
a_{1}\\
 & d_{1}\end{array}\right),\left(\begin{array}{cc}
a_{2} & 1\\
 & d_{2}\end{array}\right)\right),\quad\left(\left(\begin{array}{cc}
d_{1}\\
 & a_{1}\end{array}\right),\left(\begin{array}{cc}
d_{2} & 1\\
 & a_{2}\end{array}\right)\right).\]
These pairs are not similar (note $e_{1}\neq0$), by Lemma \ref{lem:e=00003De'},
resulting in $\bar{\Psi}$ being 2-1. 
\end{proof}
Now, the proof of Theorem \ref{thm:Invariants} is easy.

\noindent \emph{Proof of Theorem} \ref{thm:Invariants}: Let $\mathbb{F}$
be a field of characteristic $\neq2$ and $\bar{\mathbb{F}}$ its
algebraic closure. We have shown that $\bar{\Phi}:\mathcal{S}'(\bar{\mathbb{F}})/G(\bar{\mathbb{F}})\to\bar{\mathbb{F}}^{4n-3}$
is injective. Then, the map $\Phi':\mathcal{S}'(\mathbb{F})/G(\mathbb{F})\to\mathbb{F}^{4n-3}$
defined in Equation (\ref{eq:Phi'}), having the same form of $\bar{\Phi}$,
is just its restriction to $\mathcal{S}'(\mathbb{F})/G(\mathbb{F})$,
being therefore injective as well. The inclusion $\mathcal{S}'(\mathbb{F})/G(\mathbb{F})\subset\mathcal{S}'(\bar{\mathbb{F}})/G(\bar{\mathbb{F}})$
reflects the fact that two sequences in $\mathcal{S}'(\mathbb{F})$,
similar over $G(\bar{\mathbb{F}})$, are also similar over $G(\mathbb{F})$,
by the Noether-Deuring theorem (see eg. \cite{CR} p. 200). The case
of $\bar{\Psi}$ is analogous.\hfill{}$\square$

\begin{rem}
\label{rem:NoRestriction}Finally, we argue that the other cases of
semisimple non-commutative sequences in the table of Theorem \ref{thm:canonical}
can also be reconstructed by a simple modification of the above procedure.
In the case 1.b, we substitute the pair $(A_{1},A_{2})$ by the pair
$(A_{1}-A_{2},A_{1}+A_{2})$, and in the case 1.c (assuming, without
loss of generality, that $A_{1}$ is diagonal non-scalar), we perform
the substitution $(A_{1},A_{2},A_{3})\mapsto(A_{1},A_{2}+A_{3},A_{2}-A_{3})$.
In both cases we end up with a matrix in $\mathcal{S}'$, that is,
the first matrix is diagonalizable non-scalar and the first pair does
not commute. 
\end{rem}
In conclusion, the only conjugacy classes that the maps $\Phi'$ and
$\Psi'$ fail to distinguish (not counting the involution $e\leftrightarrow-e$
in the $\Psi'$ case, and after the reduction to $\mathcal{S}'$ described
above) are the following types:\[
\left(\begin{array}{cc}
a & b\\
 & a\end{array}\right),\quad\left(\begin{array}{cc}
a\\
 & a\end{array}\right),\quad a,b\in\mathbb{F}^{n}.\]
which have the same value under $\Phi'$ (note that these are not
in the domain of $\Psi'$, as they are commutative), regardless of
$b\in\mathbb{F}^{n}$.%
{}

\appendix

\section{Triangularization of singlets over $R$}

In this Appendix, we treat the triangularization problem for a single
$2\times2$ matrix over an integral domain $R$. We do not claim originality
of Proposition \ref{pro:principal} and Lemma \ref{lem:eigenvector}
below, although the author was unable to find a suitable reference.
Their exposition is mainly intended to provide self-contained proofs
of Theorems \ref{thm:Flo} and Theorem \ref{thm: pair}.

When is a single matrix with entries in $R$ triangularizable? Let
us write\begin{equation}
A=\left(\begin{array}{cc}
a & b\\
c & d\end{array}\right)\in V_{1}(R),\quad\quad g=\left(\begin{array}{cc}
x & z\\
y & w\end{array}\right)\in GL_{2}(R).\label{eq:A,g}\end{equation}
Since $g$ is invertible, so is $\de g=xw-yz$, and conjugation of
$A$ by $g^{-1}$ gives:\begin{equation}
g^{-1}\cdot A=g^{-1}\ A\ g=\frac{1}{xw-yz}\left(\begin{array}{cc}
* & bw^{2}+ezw-cz^{2}\\
cx^{2}-exy-by^{2} & *\end{array}\right).\label{eq:g.A}\end{equation}
So, triangularizing $A$ amounts to finding a solution $(x,y,z,w)\in R^{4}$
to one of the equations $cx^{2}-exy-by^{2}=0$ or $bw^{2}+ezw-cz^{2}=0$,
such that $xw-yz$ is invertible in $R$. Note that both equations
are given by the same quadratic form \[
Q(x,y):=cx^{2}-exy-by^{2}=\frac{1}{2}(x,y)Q_{A}(x,y)^{T}\]
 associated to matrix\[
Q_{A}=\left(\begin{array}{cc}
2c & -e\\
-e & -2b\end{array}\right).\]
The discriminant of $Q$ is $\delta_{Q}:=-\de Q_{A}=e^{2}+4bc=\tr^{2}A-4\de A$,
and coincides precisely with the discriminant $\delta_{A}$ of the
characteristic polynomial of $A$. %
{} 

Equation (\ref{eq:g.A}) provides a necessary condition for triangularization:
If $A$ is triangularizable over $R$, then its eigenvalues lie in
$R$. Indeed, if $g^{-1}Ag=T$ for some $g\in GL_{2}(R)$ and some
upper triangular matrix $T$, $\delta_{A}=\delta_{T}=(\lambda_{1}-\lambda_{2})^{2}$
is a square in $R$, where $\lambda_{1},\lambda_{2}\in R$ are the
diagonal elements of $T$, which are also the eigenvalues of $A$. 

A necessary and sufficient condition is the following.

\begin{prop}
\label{pro:principal}Let $A\in V_{1}(R)$ be a $2\times2$ matrix
over an integral domain $R$. $A$ is triangularizable if and only
if it has an \emph{eigenvector} of the form $(x,y)\in R^{2}$, such
that $xR+yR=R$ (in particular, the ideal $(x,y)$ is principal).
\end{prop}
\begin{proof}
The equation $Ag=gT$, for some invertible $g\in G(R)$ and upper
triangular $T$, means that the first column of $g$ is an eigenvector
$(x,y)$ of $A$, so that $x,y\in R$ and there exist $w,z\in R$
(forming the second column of $g$) so that $xw-yz$ is a unit. In
particular, $xR+yR=R$. Conversely, let $A$ be as in (\ref{eq:A,g}),
with an eigenvector $(x,y)\in R^{2}$ verifying $xR+yR=R$. From simple
computations the eigenvalues of $A$ (both in $R$) are $\lambda_{1}=\frac{a+d+r}{2}$
and $\lambda_{2}=\frac{a+d-r}{2}$, where $r$ is a square root of
$\delta_{A}=e^{2}-4bc$, and the respective eigenvectors are $v_{1}=(\lambda_{1}-d,c)$
and $v_{2}=(\lambda_{2}-d,c)$ (both in $R^{2}$). So, without loss
of generality, let $(x,y)$ be in the eigenspace of $v_{1}$: $(x,y)=\alpha(\lambda_{1}-d,c)=\alpha(\frac{e+r}{2},c)$,
for some nonzero $\alpha$ in the field of fractions of $R$. By hypothesis,
there exist $z,w\in R$ verifying $xw-yz=1$. Thus, using the invertible
matrix $g$ with columns $(x,y)$ and $(z,w)$, we have that $g^{-1}Ag$
is upper triangular (with $\lambda_{1}$ and $\lambda_{2}$ in the
main diagonal). Indeed, by Equation (\ref{eq:g.A}) its lower left
entry is $cx^{2}-exy+-by^{2}=\frac{\alpha^{2}}{4}c(r^{2}-e^{2}-4bc)=0$.
\end{proof}
The following fact is used in the proof of Theorem \ref{thm: pair}. 

\begin{lem}
\label{lem:eigenvector}Fix nonzero elements $x,y$ in a field $\mathbb{F}$.
If $A\in V_{1}(\mathbb{F})$ is nondegenerate ($\delta_{A}\neq0$)
and verifies $cx^{2}-exy-by^{2}=0$, then $A$ is \emph{diagonalizable}
and one of its eigenvectors is $(x,y)$. In particular, there is another
eigenvector $(z,w)\in\mathbb{F}^{2}$ with $wx-yz\neq0$.
\end{lem}
\begin{proof}
To satisfy $cx^{2}-exy-by^{2}=0$, the triple $(b,e,c)$ must be a
linear combination of the vectors $(-x,y,0)$ and $(0,x,y)$. So,
the matrix $A$ is given (uniquely up to the addition of a scalar)
by the triple $(b,e,c)=(-zx,zy+wx,wy)$, for some $z,w\in\mathbb{F}$.
%
{}Then, a simple computation shows that the discriminant of $A$ is
a square in $\mathbb{F}$: $\delta_{A}=(wx-zy)^{2}$. So $A$ is triangularizable
over $\mathbb{F}$. Moreover, its eigenvalues are easily checked to
be $\lambda_{1}=wx$ and $\lambda_{2}=yz$. Since, by hypothesis $\delta_{A}\neq0$,
the eigenvalues are distinct, so $A$ is diagonalizable over $\mathbb{F}$.
Also, the eigenvectors are multiples of $(x,y)$ and $(z,w)$, respectively.
\end{proof}
{}%
{}

\section{Canonical Forms and Similarity of Commuting Sequences}

For completeness, we include the following well known description
of all canonical forms of commuting sequences over an algebraically
closed field $\bar{\mathbb{F}}$. 

\begin{thm}
\label{cor:CommNF}A matrix sequence of length $n$ with coefficients
in $\bar{\mathbb{F}}$ is commutative if and only if it is similar
to a matrix sequence in one of the forms (called diagonal or triangular
forms, respectively (both forms include the scalar sequences)): \[
\mathcal{A}=\left(\begin{array}{cc}
a\\
 & d\end{array}\right),\quad\mathcal{B}=\left(\begin{array}{cc}
a & b\\
 & a\end{array}\right),\quad a,d,b\in\bar{\mathbb{F}}^{n}.\]

\end{thm}
\begin{proof}
The sufficiency is clear by Lemma \ref{lem:comm} (condition (\ref{eq:comm})
is satisfied because $b=0$ (resp. $e=0$) in the first (resp. second)
case). Conversely, since $\mathcal{A}$ is commutative it is triangularizable,
by Corollary \ref{cor:comm}. Let $k\geq1$ be the smallest integer
such that $A_{k}$ is non-scalar. By an appropriate conjugation, one
can assume that either $A_{k}$ is diagonal or it is written as a
single Jordan block (here we are using the assumption on $\bar{\mathbb{F}}$).
Then lemma \ref{lem:comm} implies that $\mathcal{A}$ is either in
diagonal or in triangular form, as wanted.
\end{proof}
The following consequence is clear.

\begin{cor}
Let $\mathcal{A}$ and $\mathcal{A}'$ be commutative, both of the
same form as described in the Theorem above. If $\mathcal{A}$ is
in diagonal form, then $\mathcal{A}\sim\mathcal{A}'$ if and only
if $a=a'$ and $d=d'$ or $a=d'$ and $d=a'$. If $\mathcal{B}$ is
in triangular form then $\mathcal{B}\sim\mathcal{B}'$ if and only
if $a=a'$ and $b=\lambda b'$ for some $\lambda\in\bar{\mathbb{F}}\setminus\left\{ 0\right\} $.
\end{cor}
%
{}%
{}%
{}

\end{comment}
{}

\bibitem[Paz]{Paz}A. Paz, \emph{An application of the Cayley-Hamilton
theorem to matrix polynomials in several variables}, Linear and Multilinear
Algebra \textbf{15} (1984) 161-170.

\bibitem[Pap]{Pap}C. J. Pappacena, \emph{An upper bound for the length
of a finite-dimensional algebra}, J. Algebra \textbf{197} (1997) 535--545.

\bibitem[Pr]{Pr}C. Procesi, \emph{The Invariant theory of $n\times n$
matrices}, Advances in Math. \textbf{19} (1976) 306-381.

\bibitem[Ra]{Ra}Yu. Razmyslov, \emph{Identities with trace in full
matrix algebras over a field of characteristic zero}, Izv. Akad. Nauk
SSSR Ser. Mat. \textbf{38} (1974), 723--756.

\bibitem[RR]{RR}H. Radjavi and P. Rosenthal, Simultaneous triangularization.
Universitext. Springer-Verlag, 2000.
\end{thebibliography}

\end{document}